\documentclass[11pt]{amsart}
\usepackage[english]{babel}
\usepackage[latin1]{inputenc}
\usepackage[dvips,final]{graphics}
\usepackage{amsmath,amsfonts,amssymb,amsthm,amscd,array,stmaryrd,mathrsfs}
\usepackage{pstricks}
 \usepackage[all]{xy}
\usepackage{graphicx}
\usepackage{color}              % Need the color package
\usepackage{epsfig}
\usepackage{psfrag}
\usepackage{epstopdf}

\usepackage{textcomp}

\let\ssection=\section
\renewcommand{\section}{\setcounter{equation}{0}\ssection}

\theoremstyle{plain}
\newtheorem{thm}{Theorem}%[section]
\newtheorem{lem}{Lemma}[section]
\newtheorem{cor}[lem]{Corollary}
\newtheorem{prop}[lem]{Proposition}

\theoremstyle{definition}
\newtheorem{defi}[lem]{Definition}
\newtheorem{rem}[lem]{Remark}
\newtheorem{ex}[lem]{Example}

\newcommand{\R}{\mathbb{R}}
\newcommand{\Z}{\mathbb{Z}}
\newcommand{\C}{\mathbb{C}}

\newcommand{\bbH}{\mathbb{H}}
\newcommand{\bbM}{\mathbb{M}}
\newcommand{\bbO}{\mathbb{O}}
\newcommand{\bbS}{\mathbb{S}}
\newcommand{\bbK}{\mathbb{K}}

\newcommand{\tA}{\mathbb{K}\left[(\Z_2)^n\right]}

\newcommand{\Cc}{\mathcal{C}}
\newcommand{\A}{\mathcal{A}}
\newcommand{\I}{\mathcal{I}}

\newcommand{\cN}{\mathcal{N}}

\newcommand{\Cl}{\mathit{C}\ell}
\newcommand{\GL}{\mathrm{GL}}

\newcommand{\Lc}{\mathcal{L}} 
 
\newcommand{\Zc}{\mathcal{Z}}

\newcommand{\gS}{\mathfrak{S}}

\newcommand{\half}{\frac{1}{2}}

\def\a{\alpha}
\def\b{\beta}
\def\d{\delta}

\def\g{\gamma}

\def\r{\rho}
\def\s{\sigma}

\def\l{\lambda}

\begin{document}

\title[A series of algebras generalizing the octonions...]{A series of algebras
generalizing the octonions and Hurwitz-Radon identity}

\author{Sophie Morier-Genoud
\and
Valentin Ovsienko}

\address{
Sophie Morier-Genoud,
Institut Math\'e\-ma\-tiques de Jussieu,
UMR 7586,
Universit\'e Pierre et Marie Curie Paris VI,
4 place Jussieu, case 247,
75252 Paris Cedex 05, France}

\address{
Valentin Ovsienko,
CNRS,
Institut Camille Jordan,
Universit\'e Claude Bernard Lyon~1,
43 boulevard du 11 novembre 1918,
69622 Villeurbanne cedex,
France}

\email{sophiemg@math.jussieu.fr,
ovsienko@math.univ-lyon1.fr}

\date{}

\begin{abstract}
We study non-associative twisted group algebras over $(\Z_2)^n$
with cubic twisting functions.
We construct a series of algebras that
extend the classical algebra of octonions
in the same way as the Clifford algebras extend the algebra of quaternions.
We study their properties, give several equivalent definitions and
prove their uniqueness within some natural assumptions.
We then prove a simplicity criterion.

We present two applications of the constructed algebras and the developed technique.
The first application is a simple explicit formula for the following famous square identity:
$(a_1^2+\cdots+a_{N}^2)\,(b_1^2+\cdots+b_{\rho(N)}^2)=
c_1^2+\cdots+c_{N}^2$, where $c_k$ are bilinear functions
of the $a_i$ and $b_j$ and where $\rho(N)$ is the Hurwitz-Radon function.
The second application is the relation to Moufang loops and, in particular,
to the code loops.
To illustrate this relation, we provide an explicit coordinate formula for the factor set
of the Parker loop.
\end{abstract}

\maketitle

{\bf Key Words}: 
Graded commutative algebras,
non-associative algebras,
Clifford algebras,
octonions,
square identities,
Hurwitz-Radon function,
code loop, Parker loop.

\medskip

{\bf Mathematics Subject Classification (2010)}:
16W50, 15A66,  11E25, 94B25.

\thispagestyle{empty}

%\vspace{-1cm}

%\newpage

\tableofcontents

%%%%%%%%%%%%%%%%%%%%%%%%%%%%%%%%%
\section{Introduction}
%%%%%%%%%%%%%%%%%%%%%%%%%%%%%%%%%

The starting idea of this work is the following naive question:
\textit{is there a natural way to multiply $n$-tuples of $0$ and $1$}?

Of course, it is easy to find such algebraic structures.
The abelian group $\left(\Z_2\right)^n$ provides such a multiplication,
but the corresponding group algebra $\bbK\left[(\Z_2)^n\right]$, 
over any field of scalars $\bbK$,
is not a simple algebra. 
A much more interesting algebraic structure on
$\bbK\left[(\Z_2)^n\right]$ is given by the twisted product 
\begin{equation}
\label{TwistProd}
u_x\cdot{}u_y=\left(-1\right)^{f(x,y)}
u_{x+y},
\end{equation}
where $x,y\in\left(\Z_2\right)^n$
and $f$ is a two-argument function on $\left(\Z_2\right)^n$
with values in $\Z_2\cong\{0,1\}$.
We use the standard notations $u_x$ for the element of $\bbK\left[(\Z_2)^n\right]$
corresponding to $x\in\left(\Z_2\right)^n$.
The only difference between the above product and
that of the group algebra $\bbK\left[(\Z_2)^n\right]$ is
the sign.
Yet, the structure of the algebra changes completely.
Throughout the paper the ground field $\bbK$ is assumed to be $\R$ or $\C$ 
(although many results hold for an arbitrary field of characteristic $\not=2$).

Remarkably enough, the classical Clifford algebras can be obtained
as twisted group algebras.
The first example is
the algebra of quaternions, $\bbH$.
This example was found by many authors but probably first in \cite{Lyc}.
The algebra $\bbH$ is a twisted $(\Z_2)^2$-algebra.
More precisely, consider the 4-dimensional vector space over $\R$ spanned by
$(0,0),\,(0,1),\,(1,0)$ and $(1,1)$ with the multiplication:
$$
u_{\left(x_1,x_2\right)}\cdot{}u_{\left(y_1,y_2\right)}=
\left(-1\right)^{x_1y_1+x_1y_2+x_2y_2}
u_{\left(x_1+y_1,\,x_2+y_2\right)}.
$$
It is easy to check that
the obtained twisted $(\Z_2)^2$-algebra is, indeed, isomorphic to $\bbH$,
see also~\cite{OM} for a different grading on the quaternions (over $(\Z_2)^3$).

Along the same lines, a Clifford algebra with $n$ generators,
is a $\left(\Z_2\right)^n$-graded algebra, see \cite{AM1}.
The (complex) Clifford algebra $\Cl_n$ is
isomorphic to the twisted group algebras over $\left(\Z_2\right)^n$
with the product
\begin{equation}
\label{Cliffunet}
u_{(x_1,\ldots,x_n)}\cdot{}u_{(y_1,\ldots,y_n)}=
\left(-1\right)^
{\sum_{1\leq{}i \leq{}j\leq{}n}x_iy_j}
u_{(x_1+y_1,\ldots,x_n+y_n)},
\end{equation}
where $(x_1,\ldots,x_n)$ is an $n$-tuple of $0$ and $1$.
The above twisting function is bilinear and therefore is a
2-cocycle on $\left(\Z_2\right)^n$.
The real Clifford algebras $\Cl_{p,q}$ are also twisted group algebras over $\left(\Z_2\right)^n$,
where $n=p+q$.
The twisting function~$f$ in the real case contains an extra term
$\sum_{1\leq{}i\leq{}p}x_iy_i$ corresponding to the signature (see Section \ref{SigSec}).

\begin{figure}[hbtp]
\begin{center}
\psfragscanon
\includegraphics[width=6cm]{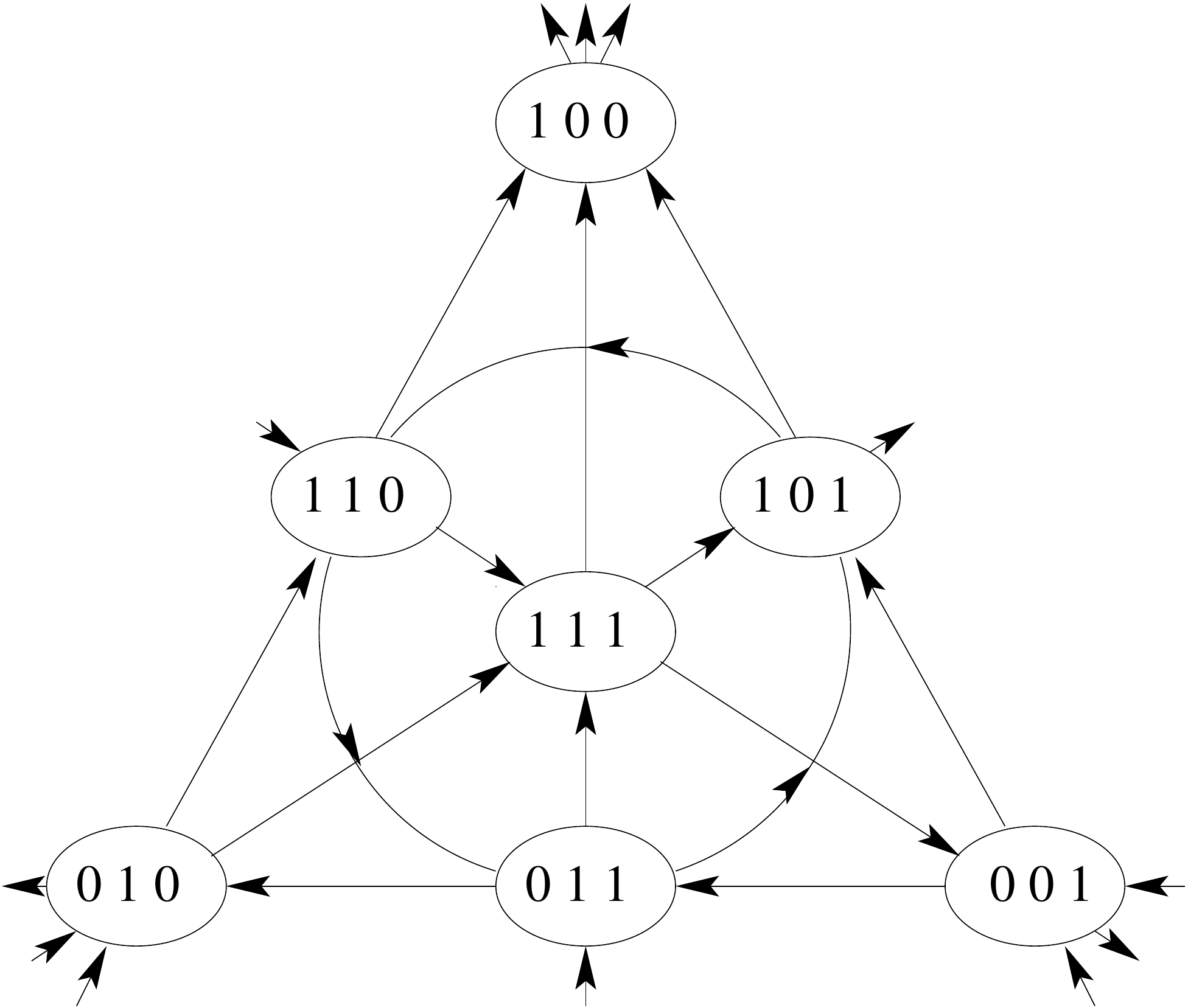}
\end{center}
\caption{$\left(\Z_2\right)^3$-grading on the octonions.}
\label{OandM}
\end{figure}

The algebra of octonions $\bbO$ can also be viewed as
a twisted group algebra \cite{AM}.
It is isomorphic to $\R\left[(\Z_2)^3\right]$
equipped with the following product:
$$
u_{(x_1,x_2,x_3)}\cdot{}u_{(y_1,y_2,y_3)}=
\left(-1\right)^{
\left(x_1x_2y_3+x_1y_2x_3+y_1x_2x_3+
\sum_{1\leq{}i\leq{}j\leq3}\,x_iy_j\right)}
u_{(x_1+y_1,\,x_2+y_2,\,x_3+y_3)}.
$$
Note that the twisting function in this case is a polynomial of degree 3,
and does not define a 2-cocycle.
This is equivalent to the fact that the algebra $\bbO$ is not associative.
The multiplication table on $\bbO$ is usually represented by the
Fano plane.
The corresponding $\left(\Z_2\right)^3$-grading is given
in Figure \ref{OandM}.
We also mention that different group gradings on $\bbO$ were studied in \cite{Eld},
we also refer to \cite{Bae} for a survey on the octonions and Clifford algebras.

In this paper, we introduce two series of complex algebras, $\bbO_{n}$ and $\bbM_n$,
and of real algebras, $\bbO_{p,q}$ and $\bbM_{p,q}$.
The series $\bbO_n$ and $\bbO_{p,q}$ generalize the algebra of octonions 
in a similar way as the Clifford
algebras generalize the algebra of quaternions.
The situation can be represented
by the following diagram
$$
\label{SubmodDiagramm}
\begin{CD}
@.@. \vdots@. \vdots \\
@.@. @AAA@AAA \\
@.@. \Cl_{4}@.\bbO_{5} \\
@.@.@AAA@AAA\\
@.@.\Cl_{3}@.\bbO_{4}\\
@.@.@AAA @AAA \\
\R@> >>\C @> >> \bbH @> >> \bbO @> >>\mathrm{\bbS}@> >>\cdots
\end{CD}
$$
where the horizontal line represents the Cayley-Dickson procedure
(see, e.g., \cite{Bae,ConSmi}),
in particular, $\bbS$ is the 16-dimensional algebra of sedenions.
The algebra $\bbM_n$ ``measures'' the difference between
$\bbO_n$ and~$\Cl_n$.

The precise definition is as follows.
The (complex) algebras $\bbO_{n}$ are twisted group algebras $\bbK\left[(\Z_2)^n\right]$ with the
product (\ref{TwistProd}), given by the function
\begin{equation}
\label{NashProd}
f_\bbO(x,y)=\sum_{1\leq{}i<j<k\leq{}n}
\left(
x_ix_jy_k+x_iy_jx_k+y_ix_jx_k
\right)+
\sum_{1\leq{}i\leq{}j\leq{}n}\,x_iy_j,
\end{equation}
for arbitrary $n$.
The algebras $\bbM_n$ are defined by the twisting function
\begin{equation}
\label{ForgProd}
f_\bbM(x,y)=\sum_{1\leq{}i<j<k\leq{}n}
\left(
x_ix_jy_k+x_iy_jx_k+y_ix_jx_k
\right),
\end{equation}
which is just the homogeneous part of degree 3 of the function $f_\bbO$
(i.e., with the quadratic part removed).
In the real case, one can again add the signature term
$\sum_{1\leq{}i\leq{}p}x_iy_i$, which
only changes the square of some generators,
and define the algebras $\bbO_{p,q}$ and $\bbM_{p,q}$.

The function $f_\bbO$ is a straightforward generalization of the
twisting function corresponding to the octonions.
In particular, the algebra $\bbO_3$ is just the complexified octonion algebra
$\bbO\otimes{}\C$.
In the real case, $\bbO_{0,3}\cong\bbO$,
the algebras $\bbO_{3,0}\cong\bbO_{2,1}\cong\bbO_{1,2}$
are isomorphic to another famous algebra called the algebra of split-octonions.
The first really interesting new example is the algebra $\bbO_5$
and its real forms $\bbO_{p,q}$ with $p+q=5$.

The algebras $\bbO_n$ and $\bbM_n$ are not associative,
moreover, they are not alternative.
It turns out however, that these algebras
have nice properties similar to those of the octonion algebra
and of the Clifford algebras at the same time.

As an ``abstract algebra'', $\bbO_n$ can be defined 
in a quite similar way as the Clifford algebras.
The algebra~$\bbO_n$ has
$n$ generators $u_1,\cdots,u_n$ such that $u_i^2=-1$ and
\begin{equation}
\label{RugaetMenia}
u_i\cdot{}u_j=-u_j\cdot{}u_i,
\end{equation}
respectively, together with the 
antiassociativity relations 
\begin{equation}
\label{RugaetMeniaBis}
u_i\cdot(u_j\cdot{}u_k)=-(u_i\cdot{}u_j)\cdot{}u_k,
\end{equation}
for $i\not=j\not=k$.
We will show that {\it the algebras $\bbO_n$
are the only algebras with $n$ generators $u_1,\cdots,u_n$
satisfying (\ref{RugaetMenia}) and (\ref{RugaetMeniaBis}) and such that
any three monomials  $u,v,w$ either associate or antiassociate
independently of the order of $u,v,w$}.

The relations of higher degree are then calculated inductively using the following
simple ``linearity law''.
Given three monomials $u,v,w$, then 
$$
u\cdot(v\cdot{}w)=(-1)^{\phi(\deg{u},\deg{v},\deg{w})}\,(u\cdot{}v)\cdot{}w,
$$
where $\phi$ is the trilinear function uniquely defined by the
above relations of degree~3, see Section \ref{GenRelSect} for the details.
For instance, one has
$
u_i\cdot((u_j\cdot{}u_k)\cdot{}u_\ell)=(u_i\cdot(u_j\cdot{}u_k))\cdot{}u_\ell,
$
for $i\not=j\not=k\not=\ell,$ etc.

The presentation of $\bbM_n$ is exactly the same as above, except that
the generators of $\bbM_n$ commute.
We will prove two classification results characterizing the algebras $\bbO_n$ and $\bbM_n$
algebras in an axiomatic way.

Our main tool is the notion of \textit{generating function}.
This is a function in one argument $\a: (\Z_2)^n\rightarrow \Z_2$ that encodes
the structure of the algebra.
Existence of a generating function is a strong condition.
This is a way to distinguish the series $\bbO_n$ and $\bbM_n$
from the classical Cayley-Dickson algebras.

The main results of the paper contain four theorems and
their corollaries.

\begin{enumerate}
\item
Theorem \ref{Thmalpha} states that the generating function
determines a (complex) twisted group algebra completely.
\item
Theorem~\ref{AlphMainTh} is a general characterization of non-associative
twisted group algebras over $(\Z_2)^n$ with symmetric non-associativity
factor, in terms of generating functions.
\item
Theorem \ref{SimProp} answers the question for which $n$ (and $p,q$)
the constructed algebras are simple.
The result is quite similar to that for the Clifford algebras, except that
the algebras $\bbO_n$ and $\bbM_n$ degenerate for one value of $n$
over 4 and not 1 over 2 as~$\Cl_n$.
\item
Theorem \ref{Solcx} provides explicit formul{\ae} of the
Hurwitz-Radon square identities.
The algebras $\bbO_n$ (as well as $\bbM_n$) are not composition algebras.
However, they have natural Euclidean norm $\cN$.
We obtain a necessary and sufficient condition
for elements $u$ and~$v$ to satisfy
$
\cN(u\cdot{}v)=\cN(u)\,\cN(v).
$
Whenever we find two subspaces $V,W\subset\bbO_n$
consisting of elements satisfying this condition,
we obtain a square identity generalizing the famous ``octonionic''
8-square identity.
\end{enumerate}

The algebras $\bbO_n$ and $\bbM_n$ are
closely related to the theory of Moufang loops and, in particular,
to code loops, see \cite{Gri,DV,NV} and references therein.
Indeed, the homogeneous elements $\pm{}u_x$, where $x\in(\Z_2)^n$
form a Moufang loop of rank~$2^{n+1}$.
As an application, we show in  Section \ref{LaSec} how the famous Parker loop
fits into our framework.

Our main tools include variations on the cohomology of $\left(\Z_2\right)^n$
and the linear algebra over~$\left(\Z_2\right)^n$.
A brief account on this subject is presented in Section \ref{CSec} and in Appendix.

\medskip

\noindent \textbf{Acknowledgments}.
This work was completed at the
Mathematisches Forschungsinstitut Oberwolfach (MFO).
The first author has benefited from the award of a
\textit{Leibniz Fellowship},
the second author is also grateful to MFO for hospitality.
We are pleased to thank
Christian~Duval, Alexey Lebedev, Dimitry Leites,
John McKay and Sergei Tabachnikov
for their interest and helpful comments.
We are also grateful to anonymous referees for a number of
useful comments.

%%%%%%%%%%%%%%%%%%%%%%%%%%%%%%%%%
%%%%%%%%%%%%%%%%%%%%%%%%%%%%%%%%%
\section{Twisted group algebras over $\left(\Z_2\right)^n$}
%%%%%%%%%%%%%%%%%%%%%%%%%%%%%%%%%
%%%%%%%%%%%%%%%%%%%%%%%%%%%%%%%%%

In this section, we give the standard definition of twisted group algebra
over the abelian group $\left(\Z_2\right)^n$.
The twisting function we consider is not necessarily a 2-cocycle.
We recall the related notion of graded quasialgebra introduced in~\cite{AM}.
In the end of the section, we give a short account on the cohomology
of $\left(\Z_2\right)^n$ with coefficients in $\Z_2$.

%%%%%%%%%%%%%%%%%%%%%%%%%%%%%%%%%
\subsection{Basic definitions}\label{TheVeryFirS}
%%%%%%%%%%%%%%%%%%%%%%%%%%%%%%%%%

The most general definition is the following.
Let $(G,+)$ be an abelian group.
A \textit{twisted group algebra}
$\left(\bbK\left[G\right],\,F\right)$
is the algebra spanned by the elements
$u_x$ for $x\in{}G$ and equipped with the product 
$$
u_x\cdot{}u_y=F(x,y)\,
u_{x+y},
$$
where
$F:G\times{}G\to\bbK^*$
is an \textit{arbitrary} two-argument function
such that 
$$
F(0,.)=F(.,0)=1.
$$
The algebra $\left(\bbK\left[G\right],\,F\right)$ is always unital and it
is associative if and only if $F$ is a 2-cocycle on $G$.
Twisted group algebras are a classical subject
(see, e.g., \cite{Con,BZ} and references therein).

We will be interested in the particular case of twisted algebras over
$G=\left(\Z_2\right)^n$ and the twisting function $F$ of the form
$$
F(x,y)=\left(-1\right)^{f(x,y)},
$$
with $f$ taking values in $\Z_2\cong\{0,1\}$.
We  will denote by $\left(\tA,\,f\right)$
the corresponding twisted group algebra.
Let us stress that the function $f$ is not necessarily a 2-cocycle.

%%%%%%%%%%%%%%%%%%%%%%%%%%%%%%%%%
\subsection{Quasialgebra structure}\label{QuaSec}
%%%%%%%%%%%%%%%%%%%%%%%%%%%%%%%%%

An arbitrary twisted group algebra
$\A=\left(\tA,\,f\right)$
gives rise to two functions
$$
\b:\left(\Z_2\right)^n\times{}\left(\Z_2\right)^n\to\Z_2,
\qquad
\phi:\left(\Z_2\right)^n\times{}\left(\Z_2\right)^n\times{}\left(\Z_2\right)^n\to\Z_2
$$
such that 
\begin{eqnarray}
\label{DeffiC}
u_x\cdot{}u_y&=&
(-1)^{\b(x,y)}\,
u_y\cdot{}u_x,\\[6pt]
\label{Deffi}
u_x\cdot\left(u_y\cdot{}u_z\right)&=&
(-1)^{\phi(x,y,z)}
\left(u_x\cdot{}u_y\right)\cdot{}u_z,
\end{eqnarray}
for any homogeneous elements $u_x,u_y,u_z\in\A$.
The function $\b$ obviously satisfies the following properties:
$\b(x,y)=\b(y,x)$ and $\b(x,x)=0$.
Following \cite{AM}, we call the structure $\b,\phi$
a \textit{graded quasialgebra}.

The functions $\b$ and $\phi$ can be expressed in terms of the twisting function $f$:
\begin{eqnarray}
\label{bef}
\b (x,y)&=&f(x,y) +f(y,x),\\[6pt]
\label{phef}
\phi (x,y,z)&=&f(y,z)+f(x+y,z)+f(x,y+z)+f(x,y).
\end{eqnarray}
Note that (\ref{phef}) reads 
$$
\phi=\d{}f.
$$
In particular, $\phi$ is a (trivial) 3-cocycle.
Conversely, given the functions $\b$ and $\phi$, to what extent the corresponding
function $f$ is uniquely defined?
We will give the answer to this question in Section \ref{IsomSec}.

\begin{ex}
(a)
For the Clifford algebra $\Cl_n$ (and for
$\Cl_{p,q}$ with $p+q=n$),
the function $\b$ is bilinear:
$$
\b_{\Cl_n}(x,y)=\sum_{i\not=j}x_iy_j.
$$
The function $\phi\equiv0$ since the twisting function (\ref{Cliffunet}) is a 2-cocycle,
this is of course equivalent to the associativity property.
Every simple graded quasialgebra
with bilinear $\b$ and $\phi\equiv0$ is a Clifford algebra,
see \cite{OM}.

(b)
For the algebra of octonions $\bbO$, the function $\b$ is as follows:
$\b(x,y)=0$ if either $x=0$, or $y=0$, or $x=y$;
otherwise, $\b(x,y)=1$.
The function $\phi$ is the determinant of $3\times3$ matrices:
$$
\phi(x,y,z)=
\det\left|
x,y,z
\right|,
$$
where $x,y,z\in\left(\Z_2\right)^3$.
This function is symmetric and trilinear.
\end{ex}

\begin{rem}
The notion of graded quasialgebra was defined in \cite{AM1}
in a more general situation where $G$ is an arbitrary
abelian group and the functions 
that measure the defect of commutativity and associativity take values
in~$\bbK^*$ instead of $\Z_2$.
The ``restricted version'' we consider is very special and
this is the reason we can say much more about it.
On the other hand, many classical algebras can be treated within
our framework.
\end{rem}

%%%%%%%%%%%%%%%%%%%%%%%%%%%%%%%%%
\subsection{The pentagonal and the hexagonal diagrams}\label{DiagSect}
%%%%%%%%%%%%%%%%%%%%%%%%%%%%%%%%%

Consider any three homogeneous elements, $u,v,w\in\A$.
The functions $\b$ and $\phi$
relate the different products, $u(vw), (uv)w,\allowbreak (vu)w$, etc.
The hexagonal diagrams in Figure \ref{5and6} 
\begin{figure}[hbtp]
\includegraphics[width=12cm]{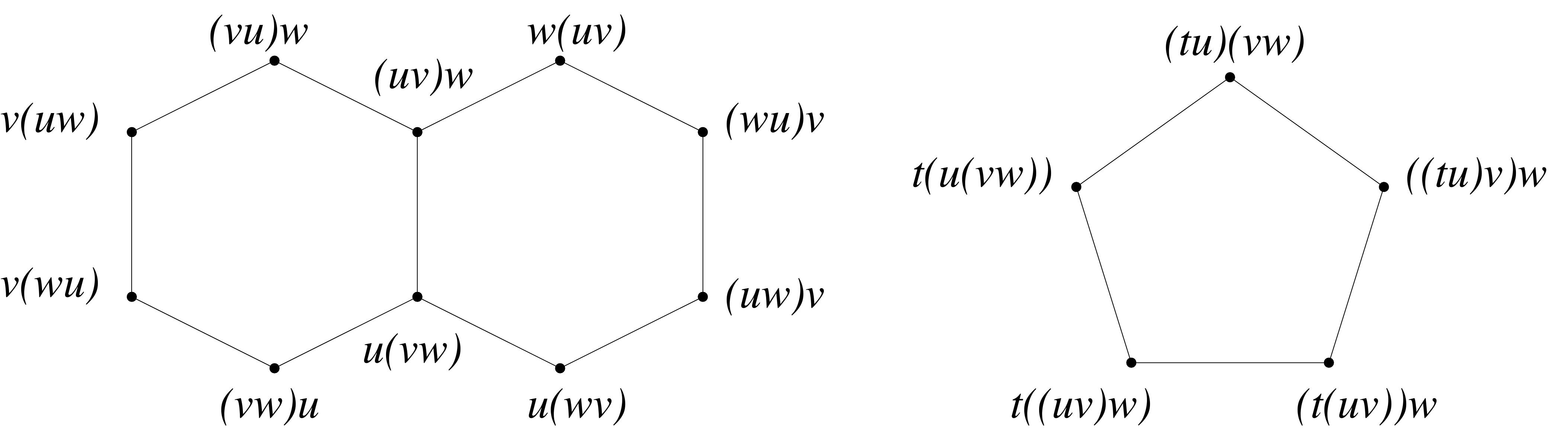}
%\end{center}
\caption{Two hexagonal and the pentagonal commutative diagrams}
\label{5and6}
\end{figure}
represent different loops in $\A$ that lead to the following identities 
\begin{equation}
\label{BetPhi}
\begin{array}{rcl}
\phi(x,y,z)+\b(x,y+z)+\phi(y,z,x)+\b(z,x)+\phi(y,x,z)+\b(x,y)&=&0\,,\\[6pt]
\phi(x,y,z)+\b(z,y)+\phi(x,z,y)+\b(z,x)+\phi(z,x,y)+\b(x+y,z)&=&0\,.
\end{array}
\end{equation}
Note that these identities can be checked directly from
(\ref{bef}) and (\ref{phef}).
In a similar way,
the products of any four homogeneous elements $t,u,v,w$, 
(see the pentagonal diagrams of Figure \ref{5and6})
 is equivalent to the condition
 \begin{equation}
\label{DeltaPhi}
\phi(y,z,t)+\phi(x+y,z,t)+\phi(x,y+z,t)+\phi(x,y,z+t)+\phi(x,y,z)=0,
\end{equation}
which is nothing but the 3-cocycle condition $\d\phi=0$.
We already knew this identity from $\phi=\d{}f$.

Let us stress on the fact that
these two commutative diagrams are tautologically satisfied
and give no restriction on $f$.

%%%%%%%%%%%%%%%%%%%%%%%%%%%%%%%%%
\subsection{Cohomology $H^*\left(\left(\Z_2\right)^n;\Z_2 \right)$}\label{CSec}
%%%%%%%%%%%%%%%%%%%%%%%%%%%%%%%%%

In this section, we recall classical notions and results on the cohomology
of $G=(\Z_2)^n$ with coefficients in $\Z_2$.

We consider the space of cochains, $\Cc^q=\Cc^q(G;\Z_2)$, consisting of (arbitrary)
maps in $q$ arguments $c:G\times\cdots\times{}G\to\Z_2$.
The usual coboundary operator $\d:\Cc^{q}\to\Cc^{q+1}$
 is defined by
$$
\d c(g_1,\ldots ,g_{q+1})= c(g_1,\ldots, g_{q})
+
\sum_{i=1}^q c(g_1,\ldots, g_{i-1},g_{i}+g_{i+1},g_{i+2},\ldots,, g_{q})
+
 c(g_2,\ldots ,g_{q+1}),
$$
for all $g_1, \ldots, g_{q+1} \in G$.
This operator satisfies $\d^2=0$.

A cochain $c$ is called \textit{$q$-cocycle} if $\d q=0$,
and called a \textit{$q$-coboundary} (or a trivial $q$-cocycle)
if $c=\d b$, for some cochain $b\in \Cc^{q-1}$.
The space of $q$-th cohomology, $H^{q}(G;\Z_2)$,
is the quotient space of $q$-cocycles modulo $q$-coboundaries.
We are particularly interested in the case where $q=1, 2$ or $3$.

A fundamental result (cf. \cite{Ade}, p.66) states that the cohomology ring
$H^*(G;\Z_2)$ is isomorphic to the algebra of polynomials
in $n$ commuting variables $e_1,\ldots,e_n$:
$$
H^*(G;\Z_2)\cong \Z_2[e_1,\ldots,e_n].
$$

The basis of $H^q(G;\Z_2)$ is given by the
cohomology classes of the following multilinear
$q$-cochains:
\begin{equation}
\label{BasCoc}
(x^{(1)},\ldots, x^{(q)})\mapsto x^{(1)}_{i_1}\cdots\, x^{(q)}_{i_q},
\qquad
i_1\leq\cdots\leq{}i_q,
\end{equation}
where each $x^{(k)}\in(\Z_2)^n$ is an $n$-tuple of 0 and 1:
$$
x^{(k)}=
(x^{(k)}_1,\ldots,x^{(k)}_n).
$$
The $q$-cocycle \eqref{BasCoc} is identified with the monomial
$e_{i_1}\cdots\,e_{i_q}$.

\begin{ex}
The linear maps $c_i(x)=x_i$, for $i=1,\ldots,n$
provide a basis of $H^{1}(G;\Z_2)$ while
the bilinear maps
$$
c_{ij}(x,y)=x_i\,y_j,
\qquad i\leq j
$$
provide a basis of the second cohomology space $H^{2}(G;\Z_2)$.
\end{ex}

%%%%%%%%%%%%%%%%%%%%%%%%%%%%%%%%%
\subsection{Polynomials and polynomial maps}\label{Poly}
%%%%%%%%%%%%%%%%%%%%%%%%%%%%%%%%%

The space of all functions on $(\Z_2)^n$ with values in $\Z_2$ is isomorphic to the
quotient space
$$
\Z_2\left[x_1,\ldots,x_n\right]\,/\,(x_i^2-x_i \,:\, i=1,\ldots ,n).
$$
A function $P:(\Z_2)^n\to\Z_2$ can be expressed as a polynomial
in $(x_1,\ldots,x_n)$,
but not in a unique way.

Throughout the paper we identify the function $P$ to the polynomial expression
in which each monomial is minimally represented (i.e. has lowest degree as possible). 
So that each function $P$ can be uniquely written in the following form 
$$
P=\sum_{k=0}^{n}\;
\sum_{1\leq{}i_1<\cdots<i_k\leq{}n}
\l_{i_1\ldots{}i_k}\,x_{i_1}\cdots\,x_{i_k},
$$
where $\l_{i_1\ldots{}i_k}\in\{0,1\}$.

%%%%%%%%%%%%%%%%%%%%%%%%%%%%%%%%%
%%%%%%%%%%%%%%%%%%%%%%%%%%%%%%%%%
\section{The generating function}
%%%%%%%%%%%%%%%%%%%%%%%%%%%%%%%%%
%%%%%%%%%%%%%%%%%%%%%%%%%%%%%%%%%

In this section we go further into the general theory of twisted group algebra
over $(\Z_2)^n$.
We define the notion of generating function.
To the best of our knowledge, this notion has never been considered.
This notion will be fundamental for us since it allows to distinguish
the Clifford algebras, octonions and the two new series we introduce in this paper
from other twisted group algebras over $(\Z_2)^n$ (such as Cayley-Dickson algebras).
The generating function contains the full information about the algebra,
except for the signature.

%%%%%%%%%%%%%%%%%%%%%%%%%%%%%%%%%
\subsection{Generating functions}\label{DefGenFn}
%%%%%%%%%%%%%%%%%%%%%%%%%%%%%%%%%
The notion of generating function makes sense for any
$G$-graded quasialgebra $\A$, over an arbitrary abelian group  $G$.
We are only interested in the case where  $\A$ is a twisted group algebra over $(\Z_2)^n$.

\begin{defi}
Given a $G$-graded quasialgebra, a function
$
\a:G\to\Z_2
$
will be called a \textit{generating function}
if the binary function $\b$ and the ternary function $\phi$
defined by (\ref{DeffiC}) and (\ref{Deffi})
are both determined by $\a$ via
\begin{eqnarray}
\label{Genalp1}
\b(x,y)&=&\a(x+y)+\a(x)+\a(y),\\[6pt]
\label{Genalp2}
\phi(x,y,z)&=&\a(x+y+z)\cr
&&+\a(x+y)+\a(x+z)+\a(y+z)\cr
&&+\a(x)+\a(y)+\a(z).
\end{eqnarray}
\noindent
Note that the identity (\ref{Genalp1}) implies that $\a$ vanishes
on the zero element $0=(0,\ldots,0)$ of $(\Z_2)^n$, because the corresponding
element $1:=u_0$ is the unit of $\A$ and therefore commutes with
any other element of $\A$.
\end{defi}

The identity (\ref{Genalp1}) means that $\b$ is the differential
of $\a$ in the usual sense of group cohomology.
The second identity (\ref{Genalp2}) suggests the operator of
``second derivation'', $\d_2$, defined by the right-hand-side, so that
the above identities then read:
$$
\b=\d\a,
\qquad
\phi=\d_2\a.
$$
The algebra $\A$ is commutative if and only if $\d\a=0$;
it is associative if and only if $\d_2\a=0$.
The cohomological meaning of the operator $\d_2$
will be discussed in Appendix.

Note also that formul{\ae} (\ref{Genalp1}) and (\ref{Genalp2})
are known in linear algebra and usually called \textit{polarization}.
This is the way one obtains a bilinear form from a quadratic one
and a trilinear form from a cubic one, respectively.

\begin{ex}
\label{ExGenHO}
(a)
The classical algebras of quaternions $\bbH$ and of octonions $\bbO$
have generating functions.
They are of the form: 
$$
\a(x)=\left\{
\begin{array}{lr}
1,&x\not=0,\\
0,&x=0.
\end{array}
\right.
$$
It is amazing that such simple functions contain the full information about
the structure of $\bbH$ and $\bbO$.

(b)
The generating function of $\Cl_n$ is as follows:
\begin{equation}
\label{ClAlp}
\a_{\Cl}(x)=\sum_{1\leq{}i\leq{}j\leq{}n}\,x_ix_j.
\end{equation}
Indeed, one checks that the binary function $\b$ defined by (\ref{Genalp1}) is exactly the
skew-symmetrization of the function
$f=\sum_{1\leq{}i \leq{}j\leq{}n}x_iy_j$.
The function $\phi$ defined by (\ref{Genalp2})
is identically zero, since $\a$ is a quadratic polynomial.
\end{ex}

The most important feature of the notion of generating function
is the following.
In the complex case, the generating function contains
the full information about the algebra.

\begin{thm}\label{Thmalpha}
If $\A$ and $\A'$ are two complex twisted group algebras with the same generating function, then 
$\A$ and $\A'$ are isomorphic as graded algebras.
\end{thm}

This theorem will be proved in Section \ref{IsomSec}.

In the real case, the generating function determines the algebra
up to the signature.

%%%%%%%%%%%%%%%%%%%%%%%%%%%%%%%%%
\subsection{The signature}\label{SigSec}
%%%%%%%%%%%%%%%%%%%%%%%%%%%%%%%%%

Consider a twisted group algebra $\A=(\bbK\left[(\Z_2)^n\right],f)$.
We will always use the following set of generators of $\A$:
\begin{equation}
\label{General}
u_i=u_{(0,\ldots,0,1,0,\ldots0)},
\end{equation}
where 1 stands at $i$-th position.
One has $u_i^2=\pm 1$,
the sign being determined by $f$.
The \textit{signature} is the data of the signs of the squares of the generators~$u_i$.

\begin{defi}
We say that the twisting functions
$f$ and $f'$
differs by a signature  if one has
\begin{equation}
\label{Signatur}
f(x,y)-f'(x,y)=x_{i_1}y_{i_1}+\cdots+x_{i_p}y_{i_p},
\end{equation}
where $p\leq{}n$ is an integer.
\end{defi}

Note that $f-f'$ as above, is a non-trivial 2-cocycle, for $p\geq 1$.
The quasialgebra structures defined by \eqref{bef} and \eqref{phef} are identically the same:
$\b=\b'$ and $\phi=\phi'$.

The signature represents the main difference between the
twisted group algebras over $\C$ and $\R$.

\begin{prop}
\label{SignProp}
If $\A=(\C\left[(\Z_2)^n\right],f)$ and $\A'=(\C\left[(\Z_2)^n\right],f')$
are complex twisted group algebras such that
$f$ and $f'$ differs by a signature, then $\A$ and $\A'$ are isomorphic
as graded algebras.
\end{prop}
\begin{proof}
Let us assume $p=1$ in \eqref{Signatur}, i.e.,
 $f(x,y)-f'(x,y)=x_{i_1}y_{i_1}$, the general case will then follow by induction.
Let $u_x$, resp. $u'_x$, be the standard basis elements of $\A$, resp. $\A'$.
Let us consider the map $\theta:\A\to\A'$ defined by
$$
\theta(u_x)=
\left\{
\begin{array}{rl}
\sqrt{-1}\,u'_x, & \text{if } \,x_{i_1}=1,\\[4pt]
u'_x, & \hbox{otherwise}.
\end{array}
\right.
$$
Note that one can write $\theta(u_x)=\sqrt{-1}^{\,x_{i_1}}u'_x$, for all $x$.
Let us show that $\theta$ is a (graded) isomorphism between $\A$ and $\A'$.
On the one hand,
$$
\theta(u_x\cdot u_y)=(-1)^{f(x,y)}\,\theta(u_{x+y})
=(-1)^{f(x,y)}\, \sqrt{-1}^{\,(x_{i_1}+y_{i_1})}\,u'_{x+y}.
$$
On the other hand,
$$
\theta(u_x)\cdot \theta(u_y)=\sqrt{-1}^{\,x_{i_1}}\,u'_x \cdot \sqrt{-1}^{\,y_{i_1}}\,u'_y
=\sqrt{-1}^{\,x_{i_1}}\,\sqrt{-1}^{\,y_{i_1}}\,(-1)^{f'(x,y)}\,u'_{x+y}.
$$
Using the following (surprising) formula
\begin{equation}
\label{Somewhat}
\frac{\sqrt{-1}^{\,x_{i_1}}\sqrt{-1}^{\,y_{i_1}}}{\sqrt{-1}^{\,(x_{i_1}+y_{i_1})}}
=(-1)^{x_{i_1}y_{i_1}},
\end{equation}
we obtain $\theta(u_x\cdot u_y)=\theta(u_x)\cdot \theta(u_y)$.

To understand \eqref{Somewhat}, beware that the power $x_{i_1}+y_{i_1}$
in the denominator is taken modulo~2.
\end{proof} 

In the real case, the algebras $\A$ and $\A'$ can be non-isomorphic
but can also be isomorphic.
We will encounter throughout this paper the algebras for which
either of this situation occurs.

%%%%%%%%%%%%%%%%%%%%%%%%%%%%%%%%%
\subsection{Isomorphic twisted algebras}\label{IsomSec}
%%%%%%%%%%%%%%%%%%%%%%%%%%%%%%%%%

Let us stress that all the isomorphisms between the twisted group algebras
we consider in this section preserve the grading.
Such isomorphisms are called \textit{graded isomorphisms}.

It is natural to ask under what condition two functions $f$ and $f'$ define
isomorphic algebras.
Unfortunately, we do not know the complete answer to this question
and give here two conditions which are sufficient but certainly not necessary.

\begin{lem}
\label{LemA}
If $f-f'=\d{}b$ is a coboundary, i.e., $b:(\Z_2)^n\to\Z_2$  is a function such that
$$
f(x,y)-f'(x,y)=b(x+y)+b(x)+b(y),
$$
then  the corresponding twisted algebras are isomorphic.
\end{lem}

\begin{proof}
The isomorphism is given by the map
$u_x\mapsto(-1)^{b(x)}\,u_x$, for all $x\in{}(\Z_2)^n$.
\end{proof}

\begin{lem}
\label{LemB}
Given a group automorphism $T:(\Z_2)^n\to{}(\Z_2)^n$, the functions $f$ and
$$
f'(x,y)=f(T(x),T(y))
$$
define isomorphic twisted group algebras.
\end{lem}

\begin{proof}
The isomorphism is given by the map
 $u_x\mapsto{}u_{T^{-1}(x)}$, for all $x\in{}(\Z_2)^n$.
\end{proof}

Note that the automorphisms of $(\Z_2)^n$ are
just arbitrary linear transformations.

We are ready to answer the question formulated in the end of Section \ref{QuaSec}.

\begin{prop}
\label{VeryFirstProp}
Given two twisted algebras $\A=(\tA,f)$ and $\A'=(\tA,f')$,
the corresponding quasialgebra structures coincide, i.e., $\b'=\b$ and $\phi'=\phi$,
if and only if
$$
f(x,y)-f'(x,y)=\d{}b(x,y)+\sum_{1\leq i \leq n}\l_i\,x_iy_i,
$$
where $b:(Z_2)^n\to\Z_2$ is an arbitrary function, and $\l_i$ are coefficients in $\Z_2$.
In particular, if $\bbK=\C$, then $\A\cong\A'$. 
\end{prop}

\begin{proof}
If the quasialgebras structures coincide, then
$\phi'=\phi$ implies $\d{}f'=\d{}f$, so that $f-f'$ is a 2-cocycle.
We use the information about the
second cohomology space $H^2((\Z_2)^n;\Z_2)$ (see Section \ref{CSec}).
Up to a 2-coboundary,
every non-trivial 2-cocycle is a linear combination of the following
bilinear maps:
$(x,y)\mapsto{}x_iy_i$, for some $i$ and $(x,y)\mapsto{}x_ky_\ell$, for some
$k<\ell$.
One deduces that $f-f'$ is of the form
$$
f(x,y)-f'(x,y)=\d{}b(x,y)+
\sum_{1\leq{}i\leq{}n}\l_i\,x_iy_i+
\sum_{k<\ell}\mu_{k\ell}\,x_ky_\ell.
$$
Since $\b'=\b$, one observes that $f-f'$ is symmetric, so that
the last summand vanishes, while the second summand is nothing but
the signature.
The isomorphism in the complex case then follows from Proposition \ref{SignProp}
and Lemma \ref{LemA}.

Conversely, if $f$ and $f'$ are related by the above expression,
then the quasialgebra structures obviously coincide.
\end{proof}

Now, we can deduce Theorem \ref{Thmalpha} as a corollary of Proposition \ref{VeryFirstProp}.
Indeed, if $\A$ and $\A'$ have the same generating function $\a$, 
then the quasialgebra structure of $\A$ and $\A'$ are the same.

%%%%%%%%%%%%%%%%%%%%%%%%%%%%%%%%%%%%%%%%%%%
\subsection{Involutions}
%%%%%%%%%%%%%%%%%%%%%%%%%%%%%%%%%%%%%%%%%%%%
Let us mention one more property of generating functions.

Recall that an \textit{involution} on an algebra $\A$ is a linear map
$a\mapsto\bar{a}$ from $\A$ to $\A$ such that
$\overline{ab}=\bar{b}\,\bar{a}$ and $\bar{1}=1$,
i.e., an involution is an anti-automorphism.
Every generating function defines a graded involution
of the following particular form:
\begin{equation}
\label{AntII}
\overline{u_x}=
(-1)^{\a(x)}\,u_x,
\end{equation}

\begin{prop}
\label{Malus'}
If $\a$ is a generating function, then the linear map defined by
formula (\ref{AntII}) is an involution.
\end{prop}

\begin{proof}
Using (\ref{DeffiC}) and (\ref{Genalp1}), one has
$$
\overline{u_xu_y}=
(-1)^{\a(x+y)}\,u_xu_y=
(-1)^{\a(x+y)+\b(x,y)}\,u_y\,u_x=
(-1)^{\a(x)+\a(y)}\,u_y\,u_x=
\overline{u_y} \,\overline{u_x}.
$$
Hence the result.
\end{proof}

In particular, the generating functions of $\bbH$ and $\bbO$, see Example \ref{ExGenHO},
correspond to the canonical involutions, i.e., to the conjugation.

%%%%%%%%%%%%%%%%%%%%%%%%%%%%%%%%%
%%%%%%%%%%%%%%%%%%%%%%%%%%%%%%%%%
\section{The series $\bbO_n$ and $\bbM_n$:
characterization}\label{TheMainS}
%%%%%%%%%%%%%%%%%%%%%%%%%%%%%%%%%
%%%%%%%%%%%%%%%%%%%%%%%%%%%%%%%%%

In this section, we formulate our first main result.
Theorem \ref{AlphMainTh}, concerns the general properties
of twisted $(\Z_2)^n$-algebras with $\phi=\d{}f$ symmetric.
This result distinguishes a class of algebras of which our algebras
of $\bbO$- and $\bbM$-series are the principal representatives.
We will also present several different ways to define the algebras 
$\bbO_n$ and $\bbM_n$, as well as of $\bbO_{p,q}$ and  $\bbM_{p,q}$.

%%%%%%%%%%%%%%%%%%%%%%%%%%%%%%%%%
\subsection{Symmetric quasialgebras}\label{SymmSec}
%%%%%%%%%%%%%%%%%%%%%%%%%%%%%%%%%

An arbitrary twisted group algebra leads to a quasialgebra structure.
One needs to assume some additional conditions
on the ``twisting'' function~$f$ in order to obtain an interesting
class of algebras.

We will be interested in the case where the function $\phi=\d{}f$,
see formula (\ref{phef}), is symmetric:
\begin{equation}
\label{SymEq}
\phi(x,y,z)=\phi(y,x,z)=\phi(x,z,y).
\end{equation}
This condition seems to be very natural:
it means that, if three elements,
$u_x,u_y$ and $u_z$ form a antiassociative triplet,
i.e., one has $u_x\cdot(u_y\cdot{}u_z)=-(u_x\cdot{}u_y)\cdot{}u_z$,
then this property is independent of the ordering of the elements
in the triplet.

An immediate consequence of the identity (\ref{BetPhi}) is that, 
if $\phi$ is symmetric, then it is completely determined by $\b$:
\begin{equation}
\label{PhiBet}
\begin{array}{rcl}
\phi(x,y,z)&=&\b(x+y,\,z)+\b(x,z)+\b(y,z)\\[4pt]
&=&\b(x,\,y+z)+\b(x,y)+\b(x,z),
\end{array}
\end{equation}
as the ``defect of linearity'' in each argument.

The following statement is our main result
about the general structure of a twisted
group algebra $\A=(\bbK[(\Z_2)^n],f)$.
We formulate this result in a slightly more general
context of $(\Z_2)^n$-graded quasialgebra.

\begin{thm}
\label{AlphMainTh}
Given a $(\Z_2)^n$-graded quasialgebra $\A$,
the following conditions are equivalent.

(i)
The function $\phi$ is symmetric.

(ii)
The algebra $\A$ has a generating function.
\end{thm}

\noindent
This theorem will be proved in Section \ref{ExProoS}.

It is now natural to ask under what condition
a function $\a:(\Z_2)^n\to\Z_2$ is a generating function
for some twisted group algebra.
The following statement provides a necessary and sufficient condition.

\begin{prop}
\label{AlpDeg3}
Given a function $\a:(\Z_2)^n\to\Z_2$,
there exists a twisted group algebra~$\A$ such that $\a$
is a generating function of $\A$, if and only if
$\a$ is a polynomial of degree $\leq3$.
\end{prop}

\noindent
This proposition will be proved in Section \ref{CubSecG}.
Furthermore, we will show in Section \ref{Unisex} that
the generating function can be chosen in a canonical way.

Theorem \ref{AlphMainTh} has a number of consequences.
In particular, it implies two more important properties of $\phi$.
The function $\phi$ is called \textit{trilinear} if it satisfies
\begin{equation}
\label{AddiT}
\phi(x+y,z,t)=\phi(x,z,t)+\phi(y,z,t),
\end{equation}
and similarly in each argument.
The function $\phi$ is \textit{alternate} if it satisfies
\begin{equation}
\label{AlT}
\phi(x,x,y)=\phi(x,y,x)=\phi(y,x,x)=0,
\end{equation}
for all $x,y\in(\Z_2)^n$.

Let us stress that an algebra satisfying \eqref{AlT}
is \textit{graded-alternative},i.e., 
$$
u_x\cdot(u_x\cdot{}u_y)=u_x^2\cdot{}u_y
\qquad
(u_y\cdot{}u_x)\cdot{}u_x=u_y\cdot{}u_x^2,
$$
for all homogeneous elements $u_x$ and~$u_y$.
This does not imply that the algebra is alternative.
Let us mention that alternative graded quasialgebras were classified in
\cite{AE}.

The following result is a consequence of Theorem \ref{AlphMainTh} and Proposition \ref{AlpDeg3}.

\begin{cor}
\label{TriCol}
If the function $\phi$ is symmetric, then $\phi$ is trilinear and alternate.
\end{cor}

\noindent
This corollary will be proved
in Section \ref{CubSecG}.

Our next goal is to study two series of algebras
with symmetric function $\phi=\d{}f$.
Let us notice that the Cayley-Dickson algebras are not of this type,
cf. \cite{AM}.

%%%%%%%%%%%%%%%%%%%%%%%%%%%%%%%%%
\subsection{The generating functions
of the algebras $\bbO_{n}$ and $\bbM_{n}$}\label{DefAlgNash}
%%%%%%%%%%%%%%%%%%%%%%%%%%%%%%%%%

We already defined the complex algebras $\bbO_{n}$ and $\bbM_{n}$,
with $n\geq3$ and the real algebra $\bbO_{p,q}$ and $\bbM_{p,q}$
see Introduction, formul{\ae} (\ref{NashProd}) and (\ref{ForgProd}).

Let us now calculate the associated function $\phi=\d{}f$
which is exactly the same for $f=f_\bbO$ or $f_\bbM$.
One obtains
$$
\phi(x,y,z)=\sum_{i\not=j\not=k}x_iy_jz_k.
$$
This function is symmetric in $x,y,z$ and Theorem \ref{AlphMainTh} implies that
the algebras $\bbO_n$ and $\bbM_{n}$ have generating functions.
The explicit formul{\ae} are as follows:
\begin{equation}
\label{NashAlp}
\a_\bbO(x)=\sum_{1\leq{}i<j<k\leq{}n}
x_ix_jx_k+
\sum_{1\leq{}i<j\leq{}n}\,x_ix_j+
\sum_{1\leq{}i\leq{}n}x_i,
\end{equation}
and
\begin{equation}
\label{NashAlpBis}
\a_\bbM(x)=\sum_{1\leq{}i<j<k\leq{}n}
x_ix_jx_k+
\sum_{1\leq{}i\leq{}n}x_i.
\end{equation}
Note that the generating functions  $\a_\bbO$ and $\a_\bbM$
 are $\gS_n$-\textit{invariant}
with respect to the natural action of the group of permutations $\gS_n$ on~$(\Z_2)^n$.

Thanks to the $\gS_n$-invariance, we can give a very simple description
of the above functions.
Denote by $|x|$ the \textit{weight} of $x\in\left(\Z_2\right)^n$
(i.e., the number of 1 entries in $x$ written as an $n$-tuple of 0 and 1).
The above generating functions, together
with that of the Clifford algebras depend only on $|x|$ and are 4-periodic:
\begin{equation}
\label{GenFuncTab}
\setlength{\extrarowheight}{3pt}
\begin{array}{|c||c|c|c|c|c|c|c|c|c}
|x| & 1 & 2 & 3 &4 &5 &6 &7 &8 &\cdots\\
\hline
\hline
\a_{\Cl} & 1 & 1 & 0 & 0 & 1 &1 &0 &0 &\cdots\\
\hline
\a_\bbO & 1 & 1 & 1 & 0 & 1& 1 & 1 & 0 &\cdots\\
\hline
\a_\bbM & 1 & 0 & 0 & 0 & 1& 0 & 0 & 0 &\cdots\\
\hline
\end{array}
\end{equation}
This table is the most simple way to use the generating function
in any calculation.
One can deduce the explicit formul{\ae} (\ref{ClAlp}), (\ref{NashAlp})
and (\ref{NashAlpBis})
directly from the table (\ref{GenFuncTab}).

%%%%%%%%%%%%%%%%%%%%%%%%%%%%%%%%%
\subsection{Characterization of the algebras of the
$\bbO$- and~$\bbM$-series}\label{UniqReSec}
%%%%%%%%%%%%%%%%%%%%%%%%%%%%%%%%%

Let us formulate two uniqueness results
that allow us to give axiomatic definitions of the introduced algebras.

Recall that the group of permutations $\gS_n$ acts on $(\Z_2)^n$
in a natural way.
We will characterize the algebras of the
$\bbO$- and~$\bbM$-series in terms of $\gS_n$-invariance.
We observe that, in spite of the fact that the functions $f_\bbO$ and $f_\bbM$
are not $\gS_n$-invariant, the corresponding algebras are.
However, we believe that $\gS_n$-invariance is a technical assumption and can be relaxed,
see Appendix for a discussion.

The first uniqueness result is formulated
directly in terms of the twisting function $f$.
We study the unital twisted algebras $\A=(\bbK\left[\left(\Z_2\right)^n\right],\,f)$
satisfying the following conditions.

\begin{enumerate}

\item
The function $f$ is a polynomial of degree 3.

\item
The algebra $\A$ is graded-alternative, see (\ref{AlT}).

\item
The set of relations between  the generators (\ref{General}) of $\A$ is
invariant with respect to the action
of the group of permutations $\gS_n$.

\end {enumerate}

Since we will use the relations of degree 2 or 3,
the condition (3) means that we assume that the generators
either all pairwise commute or all pairwise anticommute and that
one has either
$$
u_i\cdot{}(u_j\cdot{}u_k)=(u_i\cdot{}u_j)\cdot{}u_k,
\qquad
\hbox{for all}
\qquad
i\not=j\not=k,
$$
or
$$
u_i\cdot{}(u_j\cdot{}u_k)=-(u_i\cdot{}u_j)\cdot{}u_k,
\qquad
\hbox{for all}
\qquad
i\not=j\not=k.
$$

\begin{prop}
\label{Charact}
The algebras $\bbO_n$ and $\bbM_n$ are the only
twisted $\left(\Z_2\right)^n$-algebras satisfying the above three conditions.
\end{prop}

\begin{proof}
Since the algebra $\A$ is unital, we have $f(0,.)=f(.,0)=0$.
This implies that $f$ contains no constant term and no terms depending
only on $x$ (or only on $y$) variables.
The most general twisting function $f$ of degree 3 is of the form

\begin{equation*}
\begin{array}{rcll}
f(x,y)&=&\displaystyle
\sum_{i<j<k}&\left(
\l^1_{ijk}\,x_ix_jy_k+\l^2_{ijk}\,x_iy_jx_k+\l^3_{ijk}\,y_ix_jx_k\right.\\
&&&\left.+\mu^1_{ijk}\,y_iy_jx_k+\mu^2_{ijk}\,y_ix_jy_k+\mu^3_{ijk}\,x_iy_jy_k\right)\\[8pt]
&&+\displaystyle\sum_{i,j}&\nu_{ij}\,x_iy_j,
\end{array}
\end{equation*}
where $\l^e_{ijk},\mu^e_{ijk}$ and $\nu^e_{ij}$ are arbitrary coefficients 0 or 1.
Indeed, the expression of $f$ cannot contain the monomials $x_ix_jy_j$
and $x_iy_iy_j$ because of the condition (2).

By Lemma \ref{LemA}, adding a coboundary to $f$ gives an isomorphic algebra
(as $(\Z_2)^n$-graded algebras).
We may assume that for any $i<j<k$, the coefficient $\mu^1_{ijk}=0$
(otherwise, we add the coboundary of $b(x)=x_ix_jx_k$).

We now compute $\phi=\d{}f$ and obtain:
$$
\begin{array}{rcl}
\phi(x,y,z)&=&\displaystyle
\sum_{i<j<k}\left(
(\l^1_{ijk}+\mu^3_{ijk})\,x_iy_jz_k+
(\l^2_{ijk}+\mu^3_{ijk})\,x_iz_jy_k+
(\l^1_{ijk}+\mu^2_{ijk})\,y_ix_jz_k\right.\\[6pt]
&&\hskip 1cm
\left.+\l^2_{ijk}\,y_iz_jx_k+
(\l^3_{ijk}+\mu^2_{ijk})\,z_ix_jy_k+
\l^3_{ijk}\,x_iy_jz_k\right).
\end{array}
$$
We can assume that
$$
u_i\cdot{}(u_j\cdot{}u_k)=-(u_i\cdot{}u_j)\cdot{}u_k,
\qquad
i\not=j\not=k.
$$
Indeed, if $u_i\cdot{}(u_j\cdot{}u_k)=(u_i\cdot{}u_j)\cdot{}u_k$ for some values of $i,j,k$
such that $i\not=j\not=k$, then (3) implies the
same associativity relation for  all $i,j,k$.
Since $\phi$ is trilinear, this means that $\A$ is associative, so that $\phi=0$.
This can only happen if $\l^e_{ijk}=\mu^e_{ijk}=0$ for all $i,j,k$, so that $\deg\,f=2$.

In other words, we obtain a system of equations $\phi(x_i,y_j,z_k)=1$ for all $i,j,k$.
This system has a unique solution $\l^1_{ijk}=\l^2_{ijk}=\l^3_{ijk}=1$
and $\mu^2_{ijk}=\mu^3_{ijk}=0$.

Finally, if all of the generators commute, we obtain $\nu_{ij}=\nu_{ji}$, so that
$\nu_{ij}=0$ up to a coboundary, so that $f=f_\bbM$. 
If all of the generators anticommute,
again up to a coboundary, we obtain $\nu_{ij}=1$, if and only if $i<j$, so that $f=f_\bbO$.
\end{proof}

The second uniqueness result is formulated
in terms of the generating function.

\begin{prop}
\label{SimSimProp}
The algebras $\bbO_n$ and $\bbM_n$
and the algebras~$\bbO_{p,q}$ and~$\bbM_{p,q}$ with $p+q=n$,
are the only non-associative twisted $(\Z_2)^n$-algebras over
the field of scalars $\C$ or $\R$ that admit
an $\gS_n$-invariant generating function.
\end{prop}

\begin{proof}
By Proposition \ref{AlpDeg3}, we know that the generating function is a polynomial of degree~$\leq3$.
Every $\gS_n$-invariant polynomial
$\a:(\Z_2)^n\to\Z_2$ of degree $\leq3$ is a linear combination
$$
\a=\l_3\a_3+\l_2\a_2+\l_1\a_1+\l_0\a_0
$$
of the following four
functions:
$$
\a_3(x)=\sum_{1\leq{}i<j<k\leq{}n}
x_ix_jx_k,
\qquad
\a_2(x)=\sum_{1\leq{}i<j\leq{}n}\,x_ix_j,
\qquad
\a_1(x)=\sum_{1\leq{}i\leq{}n}x_i,
\qquad
\a_0(x)=1.
$$
Since $\a(0)=0$, cf. Section \ref{DefGenFn}, one obtains $\l_0=0$.
The function $\a_1$ does not contribute to the quasialgebra structure
$\b=\d\a$ and $\phi=\d_2\a$, so that $\l_1$ can be chosen arbitrary.
Finally, $\l_3\not=0$ since otherwise $\phi=0$ and the corresponding
algebra is associative.
We obtain the functions $\a_\bbO=\a_3+\a_2+\a_1$ and $\a_\bbM=\a_3+\a_1$
as the only possible $\gS_n$-invariant generating function
that define non-associative algebras.
\end{proof}

Note that relaxing the non-associativity condition $\phi\not\equiv0$, will also recover
the Clifford algebras $\Cl_n$ and $\Cl_{p,q}$ and the group algebra itself.

%%%%%%%%%%%%%%%%%%%%%%%%%%%%%%%%%
\subsection{Generators and relations}\label{GenRelSect}
%%%%%%%%%%%%%%%%%%%%%%%%%%%%%%%%%

Let us now give another definition of the complex algebras~$\bbO_n$ and~$\bbM_n$
and of the real algebras $\bbO_{p,q}$ and $\bbM_{p,q}$.
We use a purely algebraic approach and present our
algebras in terms of generators and relations.

Consider the generators \eqref{General}.
The generators of $\bbO_{p,q}$ and $\bbM_{p,q}$ square to $\pm1$.
More precisely,
\begin{equation}
\label{Squar}
u_i^2=
\left\{
\begin{array}{ll}
1, &i\leq{}p\\
-1,&\hbox{otherwise},
\end{array}
\right.
\end{equation}
where $1=u_{(0,\ldots,0)}$ is the unit.
For the complex algebras~$\bbO_n$ and~$\bbM_n$,
one can set $u_i^2=1$ for all $i$.

The rest of the relations is independent of the signature.
The main difference between the series $\bbO$ and $\bbM$ is that
the generators \textit{anticommute} in the $\bbO$-case
and \textit{commute} in the $\bbM$-case:
\begin{equation}
\label{AntiComCom}
u_i\cdot{}u_j=-u_j\cdot{}u_i
\quad
\hbox{in}\quad
\bbO_n,\;\bbO_{p,q}
\qquad
u_i\cdot{}u_j=u_j\cdot{}u_i
\quad
\hbox{in}\quad
\bbM_n,\;\bbM_{p,q}.
\end{equation}

The third-order relations are determined by the function $\phi$
and therefore these relations are the same for both series:
\begin{eqnarray}
\label{Alternass}
u_i\cdot{}(u_i\cdot{}u_j)&=&u_i^2\cdot{}u_j,\\[4pt]
\label{Antiass}
u_i\cdot{}(u_j\cdot{}u_k)&=&-(u_i\cdot{}u_j)\cdot{}u_k,
\end{eqnarray}
where $i\not=j\not=k$ in the second relation.
Note that the antiassociativity relation in (\ref{Antiass})
is the reason why the algebras from the
$\bbM$ series generated by commuting elements,
can, nevertheless, be simple.

Recall that a Clifford algebra is an algebra with $n$ anticommuting generators satisfying
the relations (\ref{Squar}) and the identity of associativity.
We will now give a very similar definition of the algebras
$\bbO_n$ and $\bbM_n$ (as well as $\bbO_{p,q}$ and $\bbM_{p,q}$).
The associativity is replaced by the identity of quasialgebra.

Define a family of algebras $\A$ with $n$ generators
$u_1,\ldots,u_n$.
Consider the monoid $X_n$ of non-associative monomials in $u_i$
and define a function $\phi:X_n\times{}X_n\times{}X_n\to\Z_2$
satisfying the following two properties:
\begin{enumerate}
\item
$
\phi(u_i,\,u_j,\,u_k)=
\left\{
\begin{array}{rl}
1,& \hbox{if}\; i\not=j\not=k,\\
0,& \hbox{otherwise}.
\end{array}
\right.
$
\medskip
\item
$\phi(u\cdot{}u',\,v,\,w)=\phi(u,\,v,\,w)+\phi(u',\,v,\,w),$ and similar in each variable.
\end{enumerate}
Such function exists and is unique.
Moreover, $\phi$ is symmetric.

Define an algebra $\A^\C$ or $\A^\R$ (complex or real),
generated by $u_1,\ldots,u_n$ that satisfies
the relations (\ref{Squar}) together with one of the following two relations.
All the generators either anticommute: $u_i\cdot{}u_j=-u_j\cdot{}u_i$, where $i\not=j$, or
commute: $u_i\cdot{}u_j=u_j\cdot{}u_i$, where $i\not=j$.

We will also assume the identity
$$
u\cdot(v\cdot{}w)=(-1)^{\phi(u,v,w)}\,(u\cdot{}v)\cdot{}w,
$$
for all monomials $u,v,w$.

\begin{prop}
\label{InvSnProp}
If the generators anticommute, then $\A^\C\cong\bbO_n$
and $\A^\R\cong\bbO_{p,q}$.
If the generators commute, then $\A^\C\cong\bbM_n$
and $\A^\R\cong\bbM_{p,q}$.
\end{prop}

\begin{proof}
By definition of $\A=\A^\C$ (resp. $\A^\R$), the elements
$$
u_{i_1\ldots{}i_k}=
u_{i_1}\cdot(u_{i_2}\cdot(\cdots(u_{i_{k-1}}\cdot{}u_{i_k})\!\cdots\!),
$$
where $i_1<i_2<\cdots<i_k$, form a basis of $\A$.
Therefore, $\dim\A=2^n$.
The linear map sending the generators of $\A$
to the generators (\ref{General}) of  $\bbO_n$ or $\bbM_n$
($\bbO_{p,q}$ or $\bbM_{p,q}$, respectively) is a homomorphism,
since the function~$\phi$ corresponding to these algebras is symmetric and trilinear.
It sends the above basis of $\A$ to that of $\bbO_n$ or $\bbM_n$
($\bbO_{p,q}$ or $\bbM_{p,q}$, respectively).
\end{proof}

%%%%%%%%%%%%%%%%%%%%%%%%%%%%%%%%%
%%%%%%%%%%%%%%%%%%%%%%%%%%%%%%%%%
\section{The series $\bbO_n$ and $\bbM_n$: properties}\label{TheMainSBis}
%%%%%%%%%%%%%%%%%%%%%%%%%%%%%%%%%
%%%%%%%%%%%%%%%%%%%%%%%%%%%%%%%%%

In this section, we study properties of the algebras of the series $\bbO$ and $\bbM$.
The main result is Theorem \ref{SimProp} providing
a criterion of simplicity.
We describe the first algebras of the series and give the list of isomorphisms
in lower dimensions.
We also define a non-oriented graph encoding the structure of the algebra.
Finally, we formulate open problems.

%%%%%%%%%%%%%%%%%%%%%%%%%%%%%%%%%
\subsection{Criterion of simplicity}
%%%%%%%%%%%%%%%%%%%%%%%%%%%%%%%%%

The most important property of the defined algebras that we
study is the simplicity.
Let us stress that 
we understand simplicity in the usual sense:
an algebra is called \textit{simple} if it contains no proper ideal.
Note that in the case of commutative associative algebras,
simplicity and division are equivalent notions,
in our situation, the notion of simplicity is much weaker.

\begin{rem}
This notion should not be confounded with the notion of graded-simple algebra.
The latter notion is much weaker and means that the algebra contains no graded ideal;
however, this notion is rather a property of the grading and not
of the algebra itself.
\end{rem}

The following statement is the second main result of this paper.
We will treat the complex and the real cases independently.

\begin{thm}
\label{SimProp}
(i)
The algebra $\bbO_n$ (resp. $\bbM_n$)
is simple if and only if $n\not=4m$ (resp. $n\not=4m+2$).
One also has
$$
\bbO_{4m}\cong\bbO_{4m-1}\oplus\bbO_{4m-1},
\qquad
\bbM_{4m+2}\cong\bbM_{4m+1}\oplus\bbM_{4m+1}.
$$

\noindent
(ii) The algebra $\bbO_{p,q}$ is simple if and only if one of the following conditions is satisfied
\begin{enumerate}
\item
$p+q\not=4m$, 
\item $p+q=4m$ and $p,q$ are odd;
\end{enumerate}
(iii) The algebra $\bbM_{p,q}$ is simple if and only if one of the following conditions is satisfied
\begin{enumerate}
\item
$p+q\not=4m+2$, 
\item $p+q=4m+2$ and $p,q$ are odd.
\end{enumerate}
\end{thm}

This theorem will be proved in Section \ref{ProoSimProp}.

The arguments developed in the proof of Theorem \ref{SimProp}
allow us to link the complex and the real algebras
in the particular cases below.
Let us use the notation $\bbO_n^\R$ and $\bbM_n^\R$ 
when we consider the algebras $\bbO_n$ and $\bbM_n$ as
$2^{n+1}$-dimensional real algebras.
We have the following statement.

\begin{cor}
\label{IsoPr}
(i)
If $p+q=4m$ and $p,q$ are odd, then $\bbO_{p,q}\cong\bbO_{p+q-1}^\R$.

(ii)
If $p+q=4m+2$ and $p,q$ are odd, then $\bbM_{p,q}\cong\bbM_{p+q-1}^\R$.
\end{cor}
\noindent
This statement is proved in Section \ref{Thmii}.

\begin{rem}
To explain the meaning of the above statement,
we notice that, in the case where the complex algebras
split into a direct sum, the real algebras can still be simple.
In this case, all the simple real algebras are isomorphic to
the complex algebra with $n-1$ generators.
In particular, all the algebras $\bbO_{p,q}$ and $\bbO_{p',q'}$
with $p+q=p'+q'=4m$ and $p$ and $p'$ odd are isomorphic
to each other (and similarly for the $\bbM$-series).
A very similar property holds for the Clifford algebras.
\end{rem}

Theorem \ref{SimProp} immediately implies the following.

\begin{cor}
\label{NoisomProp}
The algebras $\bbO_{n}$ and $\bbM_n$ with even $n$ are not isomorphic.
\end{cor}
\noindent
This implies, in particular, that the real algebras $\bbO_{p,q}$
and $\bbM_{p',q'}$ with $p+q=p'+q'=2m$ are not isomorphic.

%%%%%%%%%%%%%%%%%%%%%%%%%%%%%%%%%
\subsection{The first algebras of the series}
%%%%%%%%%%%%%%%%%%%%%%%%%%%%%%%%%

Let us consider the first examples of the introduced algebras.
It is natural to ask if some of the introduced algebras can be isomorphic to the other ones.

\begin{prop}
\label{Sporadic}
(i)
For $n=3$, one has:
$$
\bbO_{3,0}\cong\bbO_{2,1}\cong\bbO_{1,2}\not\cong\bbO_{0,3}.
$$
The first three algebras are isomorphic to the algebra of split-octonions,
while $\bbO_{0,3}\cong\bbO$.

(ii)
For $n=4$, one has:
$$
\bbO_{4,0}\cong\bbO_{2,2}\cong\bbO_{3,0}\oplus\bbO_{3,0}\,,
\qquad
\bbO_{0,4}\cong\bbO_{0,3}\oplus\bbO_{0,3}.
$$
In particular, $\bbO_{4,0}$ and $\bbO_{2,2}$ are not isomorphic to $\bbO_{0,4}$.
\end{prop}

\begin{proof}
The above isomorphisms are combination of the general
isomorphisms of type (a) and (b), see Section \ref{IsomSec}.
The involved automorphisms of $(\Z_2)^3$ and $(\Z_2)^4$ are
$$
\begin{array}{lcl}
x_1'=x_1,&&x_1'=x_1,\\[4pt]
x_2'=x_1+x_2,&&x_2'=x_1+x_2,\\[4pt]
x_3'=x_1+x_2+x_3, \;\;\;\;\;\;\;\;\;&&x_3'=x_1+x_2+x_3,\\[4pt]
&&x_4'=x_1+x_2+x_3+x_4.
\end{array}
$$
Then, the twisting functions of the above isomorphic algebras
coincide modulo coboundary.
\end{proof}

Let us notice that the very first algebras of the $\bbO$-series
are all obtained as a combination of the algebras
of octonions and split-octonions.
In this sense, we do not obtain new algebras among them.

In the $\bbM$-case, we have the following isomorphism.

\begin{prop}
\label{M3Prop}
One has
$$
\bbM_{1,2}\cong\bbM_{0,3},
$$
\end{prop}

\begin{proof}
This isomorphism can be obtained by the following automorphism of $(\Z_2)^3$.
$$
x_1'=x_1+x_2+x_3,\quad
x_2'=x_2,\quad
x_3'=x_3
$$
This algebra is not isomorphic to $\bbO_{0,3}$ or $\bbO_{3,0}$.
\end{proof}

The next algebras,
$\bbO_5$ and $\bbM_5$, as well as all of the real algebras
$\bbO_{p,q}$ and $\bbM_{p,q}$
with $p+q=5$, are not combinations of the classical algebras.
Since these algebras are simple, they are not direct sums
of lower-dimensional algebras.
The next statement shows that these algebras are not
tensor products of classical algebras.
Note that the only ``candidate'' for an isomorphism of this kind
is the tensor product of the octonion algebra and the algebra of complex
$(2\times2)$-matrices.

\begin{prop}
\label{IndProp}
Neither of the algebras $\bbO_5$ and $\bbM_5$ is isomorphic to the tensor product
of the octonion algebra~$\bbO$ and the algebra $\C[2]$ of complex
$(2\times2)$-matrices:
$$
\bbO_5\not\cong\bbO\otimes\C[2],
\qquad
\bbM_5\not\cong\bbO\otimes\C[2].
$$
\end{prop}

\begin{proof}
Let us consider the element $u=u_{(1,1,1,1,0)}$ in $\bbO_5$
and the element $u=u_{(1,1,0,0,0)}$ in $\bbM_5$.
Each of these elements has a very big centralizer $Z_u$ of $\dim{}Z_u=24$.
Indeed, the above element of $\bbO_5$
commutes with itself and with any homogeneous element $u_x$
of the weight $|x|=0,1,3,5$ as well as 6 elements such that $|x|=2$.
The centralizer $Z_u$ is the vector space spanned by these 24 homogeneous elements,
and similarly in the $\bbM_5$ case.
We will show that the algebra $\bbO\otimes\C[2]$ does not contain
such an element.

Assume, \textit{ad absurdum}, that an element $u\in\bbO\otimes\C[2]$
has a centralizer of dimension~$\geq24$.
Consider the subspace $\bbO\otimes1\oplus1\otimes\C[2]$ of the algebra $\bbO\otimes\C[2]$.
It is 12-dimensional, so that its intersection with $Z_u$ is of dimension at least 4.
It follows that $Z_u$ contains at least two independent elements of the form
$$
z_1=e_1\otimes1+1\otimes{}m_1,
\qquad
z_2=e_2\otimes1+1\otimes{}m_2,
$$
where $e_1$ and $e_2$ are pure imaginary octonions and
$m_1$ and $m_2$ are traceless matrices.

Without loss of generality, we can assume that one of
the following holds
\begin{enumerate}
\item
the generic case:
$e_1,e_2$ and
$m_1,m_2$ are linearly independent and
pairwise anticommute,
\item
$e_2=0$ and $m_1,m_2$ are linearly independent and anticommute,
\item
$m_2=0$  and $e_1,e_2$ are linearly independent and anticommute.
\end{enumerate}

We will give the details of the proof in the case (1).
Let us write
$$
u=u_0\otimes 1 +u_1\otimes m_1 +u_2\otimes m_2 +u_{12}\otimes m_1m_2 
$$
where $u_0,u_1,u_2,u_{12}\in \bbO$.

\begin{lem}
\label{TechnLem}
The element $u$ is a linear combination of the following two elements:
$$
1\otimes1,
\qquad
e_1\otimes{}m_1+e_2\otimes{}m_2-e_1e_2\otimes{}m_1m_2.
$$
\end{lem}

\begin{proof}
Denote by $[\;,\;]$ the usual commutator, one has
$$
\begin{array}{rcl}
[u,z_1]&=&[u_0,e_1]\otimes1+[u_1,e_1]\otimes{}m_1+
\left([u_2,e_1]-2u_{12}\right)\otimes{}m_2+
\left([u_{12},e_1]-2u_{2}\right)\otimes{}m_1m_2,\\[8pt]
[u,z_2]&=&[u_0,e_2]\otimes1+[u_2,e_2]\otimes{}m_2+
\left([u_1,e_2]+2u_{12}\right)\otimes{}m_1+
\left([u_{12},e_2]+2u_{1}\right)\otimes{}m_1m_2.
\end{array}
$$
One obtains $[u_0, e_1]=[u_0,e_2]=0$, so that $u_0$ is proportional to 1.
Furthermore, one also obtains $[u_1,e_1]=0$ and $[u_2,e_2]=0$ that implies
$$
u_1=\l_1\,e_1+\mu_1\,1,
\qquad
u_2=\l_2\,e_2+\mu_2\,1.
$$
The equations $[u_2,e_1]-2u_{12}=0$ and $[u_1,e_2]+2u_{12}$ give
$$
u_{12}=\l_2\,e_2e_1
\qquad\hbox{and}\qquad
u_{12}=-\l_1\,e_1e_2,
$$
hence $\l_1=\l_2$, since $e_1$ and $e_2$ anticommute by assumption.
Finally, the equations $[u_{12},e_1]-2u_{2}=0$ and $[u_{12},e_2]+2u_{1}=0$
lead to $\mu_1=\mu_2=0.$

Hence the lemma.
\end{proof}

In the case (1), one obtains a contradiction because of the following statement.

\begin{lem}
\label{TechnLemBis}
One has $\dim{}Z_u\leq22$.
\end{lem}

\begin{proof}
Lemma \ref{TechnLem} implies that the element $u$ 
belongs to a subalgebra 
$$
\C[4]=\C[2]\otimes\C[2]\subset\bbO\otimes\C[2].
$$
We use the well-known classical fact that, for
an arbitrary element $u\in\C[4]$, the dimension of the centralizer inside $\C[4]$:
$$
\{X\in\C[4]\,|\,[X,u]=0\}
$$
is at most 10 (i.e., the codimension is $\leq6$).
Furthermore, the 4-dimensional space
of the elements $e_3\otimes1$, where $e_3\in\bbO$ anticommutes with $e_1,e_2$
is transversal to~$Z_u$.

It follows that the codimension of $Z_u$ is at least 10.
Hence the lemma.
\end{proof}

The cases (2) and (3) are less involved.
In the case (2), $u$ is proportional to
$e_1\otimes1$ and one checks that $Z_u=e_1\otimes\C[2]\oplus1\otimes\C[2]$
is of dimension 8.
In the case (3), $u$ is proportional to $1\otimes{}m_1$ so that 
$Z_u=\bbO\otimes1\oplus\bbO\otimes{}m_1$ is of dimension 16.
In each case, we obtain a contradiction.
\end{proof}

%%%%%%%%%%%%%%%%%%%%%%%%%%%%%%%%%
\subsection{The commutation graph}
%%%%%%%%%%%%%%%%%%%%%%%%%%%%%%%%%

We associate a non-oriented graph, that we call the commutation graph, to every
twisted group algebra
in the following way.
The vertices of the graph are the elements of $\left(\Z_2\right)^n$.
The (non-oriented) edges $x-y$ join the elements $x$ and $y$ such that
$u_x$ and $u_y$ anticommute.

\begin{prop}
\label{graphprop}
Given a complex algebra $(\C[(Z_2)^n],f)$ with symmetric function $\phi=\d{}f$,
the commutation graph completely determines the
structure of $\A$.
\end{prop}
\begin{proof}
In the case where $\phi$ is symmetric, formula (\ref{PhiBet}) and
Proposition \ref{VeryFirstProp} imply that the graph determines the
structure of the algebra $\A$, up to signature.
\end{proof}

This means, two complex algebras, $\A$ and $\A'$ corresponding to the same
commutation graph are isomorphic.
Conversely, two algebras, $\A$ and $\A'$
with different commutation graphs,
are not isomorphic as $(\Z_2)^n$-graded algebras.
However, we do not know if there might exist an isomorphism
that does not preserve the grading.

\begin{ex}
The algebra $\bbM_{3}$ is the first non-trivial algebra of the series $\bbM_{n}$.
The corresponding commutation graph is presented in Figure \ref{M3Alg}, together
with the graph of the Clifford algebra $\Cl_3$.

\begin{figure}[hbtp]
\includegraphics[width=11cm]{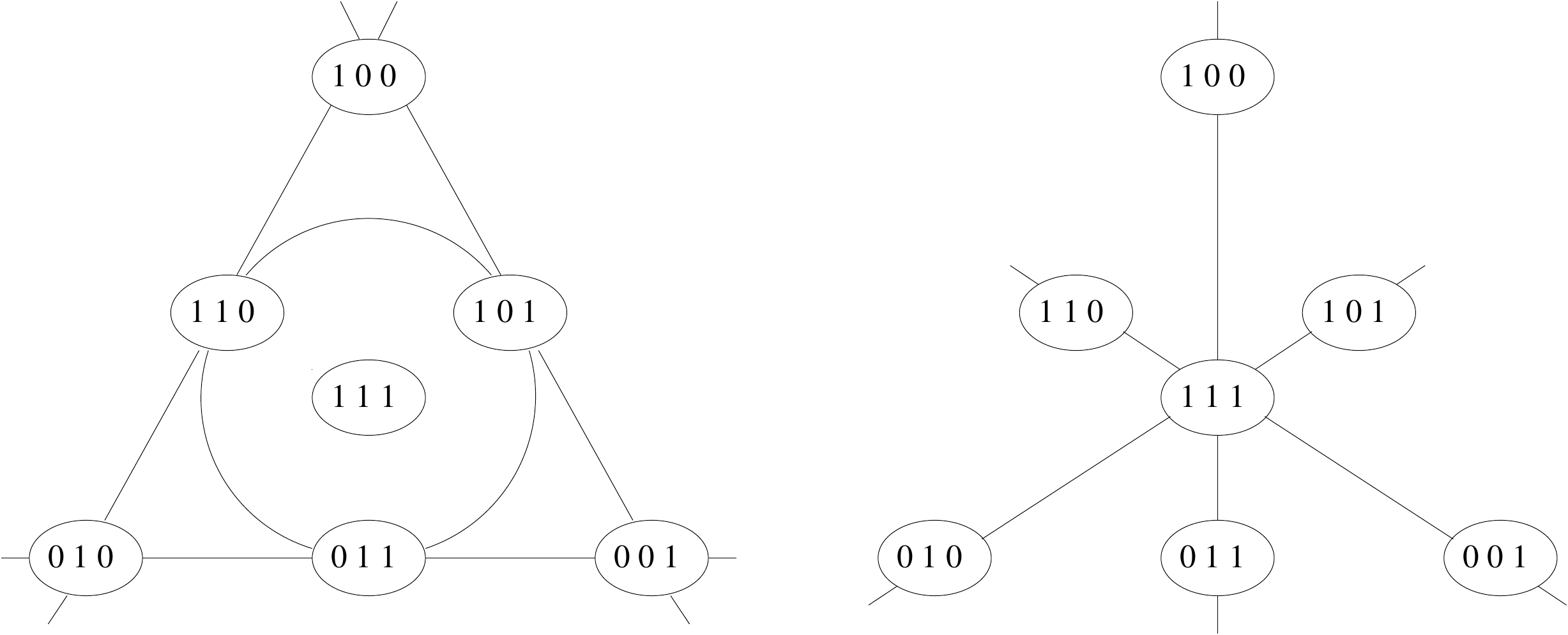}
\caption{The algebras $\Cl_3$ and $\bbM_3$.}
\label{M3Alg}
\end{figure}
\noindent

The algebra $\Cl_3$ is not simple: $\Cl_3=\C[2]\oplus\C[2]$.
It contains a central element $u_{(1,1,1)}$ corresponding to a
``singleton'' in Figure \ref{M3Alg}.

\end{ex}

\begin{rem}
(a)
The defined planar graph is \textit{dual trivalent},
that is, every edge
represented by a projective line or a circle,
see Figure \ref{M3Alg}, contains exactly 3 elements.
Indeed, any three homogeneous elements
$u_x,u_y$ and $u_{x+y}$ either commute or anticommute
with each other.
This follows from the tri-linearity of $\phi$.

(b)
We also notice that the superposition of the graphs
of $\Cl_n$ and $\bbM_n$ is precisely
the graph of the algebra $\bbO_n$.
We thus obtain the following ``formula'': $\Cl+\bbM=\bbO$.
\end{rem}

\begin{figure}[hbtp]
\includegraphics[width=11cm]{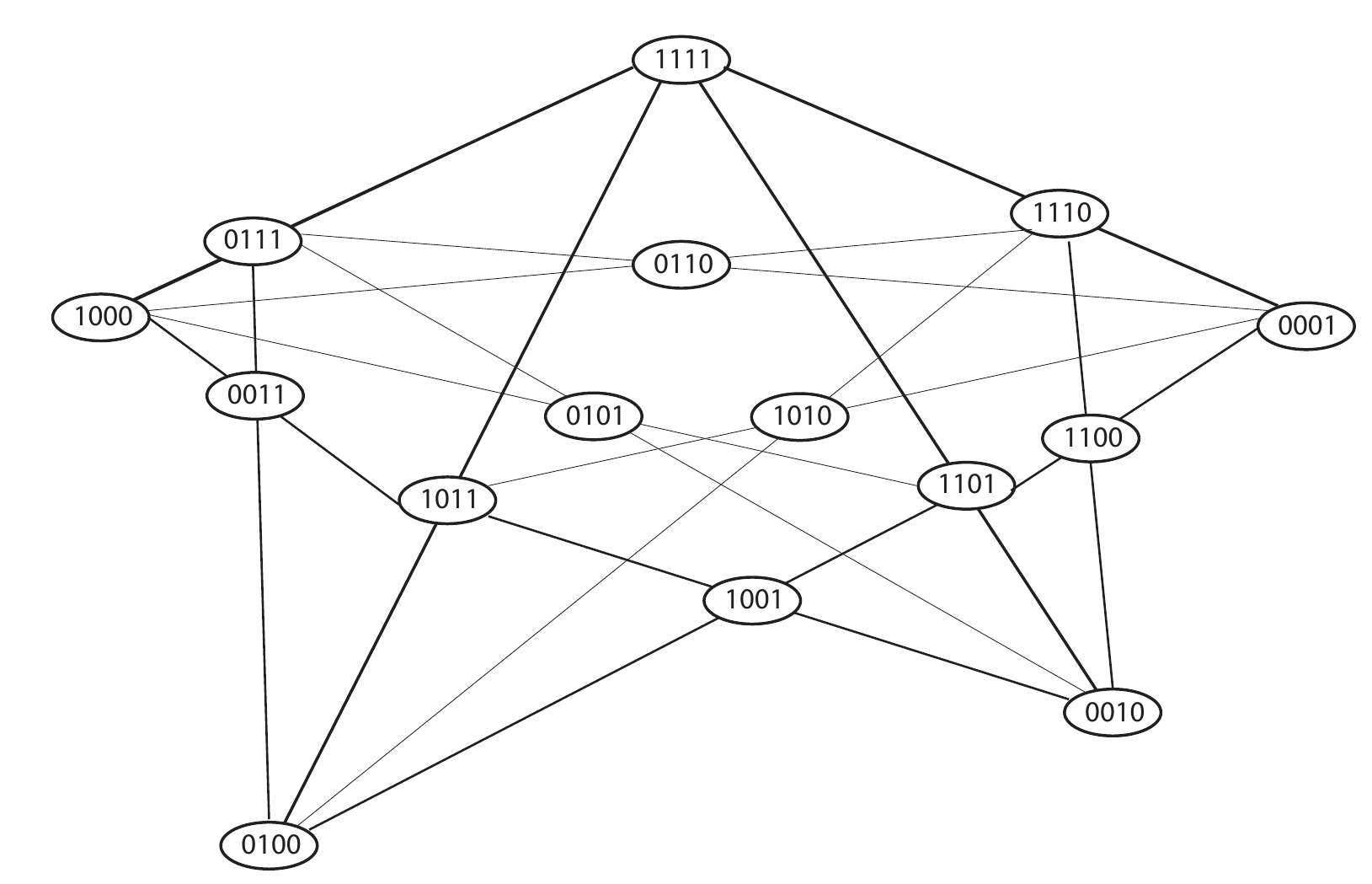}
\caption{The commutation graph of $\bbM_4$.}
\label{M4Alg}
\end{figure}

\begin{ex}
The commutation graph of the algebra $\bbM_4$
is presented in Figure \ref{M4Alg}.

The commutation graph of the Clifford algebra $\Cl_4$ is is presented in Figure \ref{C4Alg}.
Note that both algebras, $\bbM_4$ and $\Cl_4$ are simple.
The superposition of the graphs of $\bbM_4$ and $\Cl_4$ cancels
all the edges from $(1,1,1,1)$.
Therefore, the element $(1,1,1,1)$ is a singleton in the graph of
the algebra $\bbO_4$.
This corresponds to the fact that $u_{(1,1,1,1)}$ in $\bbO_4$ is a central,
in particular, $\bbO_4$ is not simple.

\begin{figure}[hbtp]
\includegraphics[width=11cm]{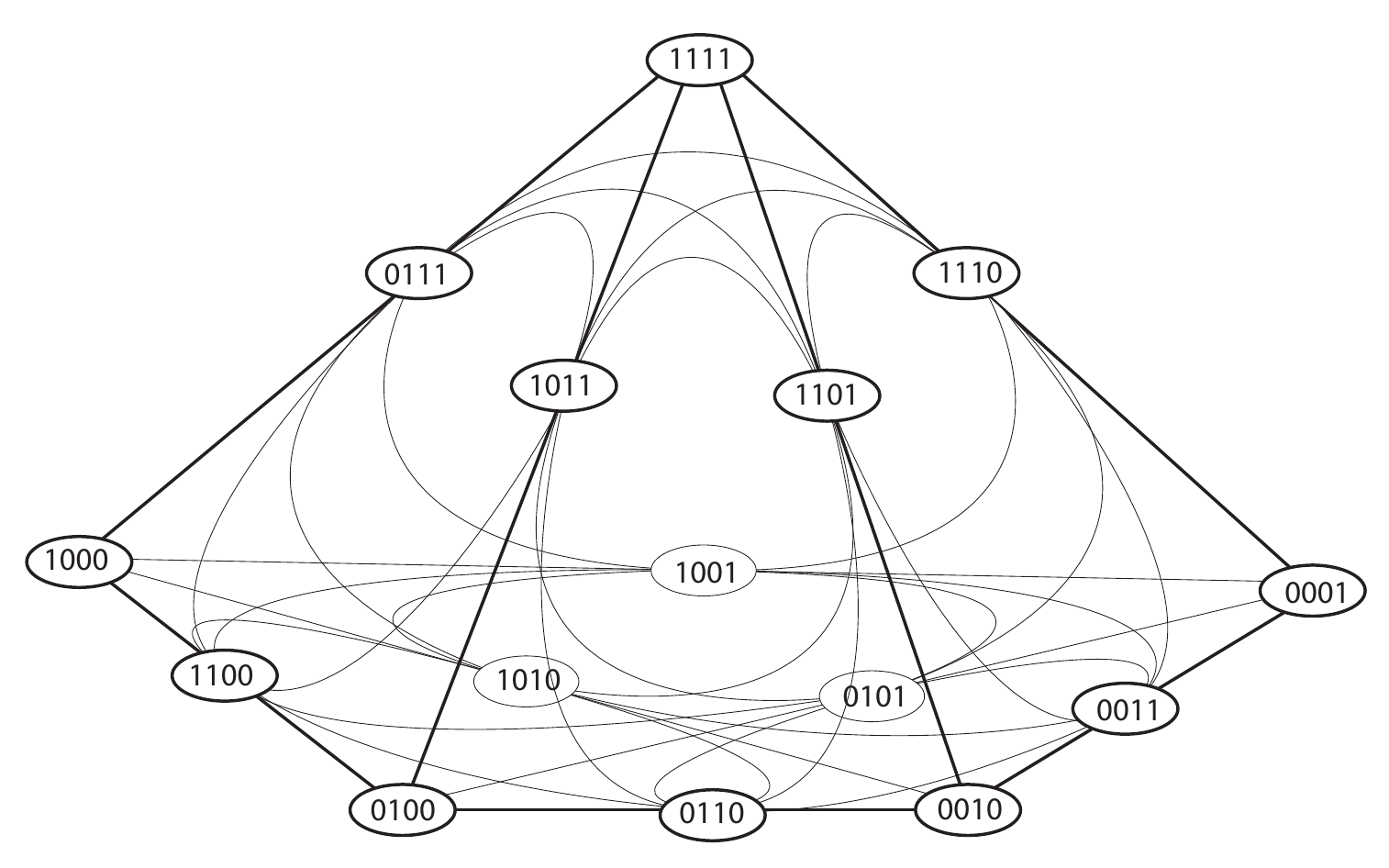}
\caption{The commutation graph of $\Cl_4$.}
\label{C4Alg}
\end{figure}

\end{ex}

The planar graph provides a nice way
to visualize the algebra $(\tA,f)$.

%%%%%%%%%%%%%%%%%%%%%%%%%%%%%%%%%
%%%%%%%%%%%%%%%%%%%%%%%%%%%%%%%%%
\section{Generating functions: existence and uniqueness}\label{AntiSex}
%%%%%%%%%%%%%%%%%%%%%%%%%%%%%%%%%
%%%%%%%%%%%%%%%%%%%%%%%%%%%%%%%%%

In this section we prove Theorem \ref{AlphMainTh} and its corollaries.
Our main tool is the notion of generating function.
We show that the structure of all the algebras we consider in this paper
is determined (up to signature) by a single function of one argument
$\a:\left(\Z_2\right)^n\to\Z_2.$
This of course simplifies the understanding of these algebras.

%%%%%%%%%%%%%%%%%%%%%%%%%%%%%%%%%
\subsection{Existence of a generating function}\label{ExProoS}
%%%%%%%%%%%%%%%%%%%%%%%%%%%%%%%%%

Given a $(\Z_2)^n$-graded quasialgebra, let us prove that
there exists a generating function $\a$ if and only
if the ternary map $\phi$ is symmetric.

The condition that $\phi$ is symmetric is of course necessary
for existence of $\a$, cf. formula (\ref{Genalp2}), let us prove that this condition is,
indeed, sufficient.

\begin{lem}
\label{dbetZero}
If $\phi$ is symmetric, then $\b$ is a 2-cocycle: $\d\b=0$.
\end{lem}
\begin{proof}
If $\phi$ is symmetric then the identity (\ref{PhiBet}) is satisfied.
In particular, the sum of the two expressions of $\phi$ gives:
$$
\b(x+y,\,z)+\b(x,\,y+z)+\b(x,y)+\b(y,z)=0
$$
which is nothing but the 2-cocycle condition $\d\b=0$.
\end{proof}

Using the information about the second cohomology
space $H^2((\Z_2)^n;\Z_2)$, as in the proof of Proposition \ref{VeryFirstProp},
we deduce that $\b$ is of the form
$$
\b(x,y)=\d\a(x,y)+
\sum_{i\in{}I}x_iy_i+
\sum_{(k,\ell)\in{}J}x_ky_\ell,
$$
where $\a:(\Z_2)^n\to\Z_2$ is an arbitrary function and
where $I$ is a subset of $\{1,\ldots,n\}$ and
$J$ is a subset of $\{(k,\ell)\,|\,k<\ell\}$.
Indeed, the second and the third terms are the most general non-trivial
2-cocycles on $(\Z_2)^n$ with coefficients in $\Z_2$.

Furthermore, the function $\b$ satisfies two properties:
it is symmetric and $\b(x,x)=0$.
The second property implies that $\b$ does not contain
the  terms $x_iy_i$.
The symmetry of $\b$ means that whenever there is a term
 $x_ky_\ell$, there is $x_\ell{}y_k$,
as well.
But, $x_ky_\ell+x_\ell{}y_k$ is a coboundary of $x_kx_\ell$.
We proved that $\b=\d\a$, which is equivalent to the identity \eqref{Genalp1}.

Finally,
using the equality \eqref{PhiBet}, we also obtain the identity \eqref{Genalp2}.

Theorem \ref{AlphMainTh} is proved.

%%%%%%%%%%%%%%%%%%%%%%%%%%%%%%%%%
\subsection{Generating functions are cubic}\label{CubSecG}
%%%%%%%%%%%%%%%%%%%%%%%%%%%%%%%%%

In this section, we prove Proposition \ref{AlpDeg3}.
We show that a function $\a:(\Z_2)^n\to\Z_2$ is a generating function of a
$(\Z_2)^n$-graded quasialgebra if and only if $\a$ is a polynomial
of degree $\leq3$.

The next statement is an application of the pentagonal diagram
in Figure~\ref{5and6}.

\begin{lem}
\label{KeyLem}
A generating function $\a:(\Z_2)^n\to\Z_2$ satisfies the
equation $\d_3\a=0$, where the map $\d_3$ is defined by
\begin{equation}
\label{DelTrequat}
\begin{array}{rl}
\d_3\a\,(x,y,z,t):=&
\a(x+y+z+t)\\[4pt]
&+\a(x+y+z)+\a(x+y+t)+\a(x+z+t)+\a(y+z+t)\\[4pt]
&+\a(x+y)\!+\!\a(x+z)\!+\!\a(x+t)\!+\!\a(y+z)\!+\!\a(y+t)\!+\!\a(z+t)\\[4pt]
&+\a(x)+\a(y)+\a(z)+\a(t).
\end{array}
\end{equation}
\end{lem}
\begin{proof}
This follows immediately from the fact that $\phi$ is a 3-cocycle:
substitute (\ref{Genalp2}) to the equation $\d\phi=0$ to obtain
$\d_3\a=0$.
\end{proof}

The following statement characterizes polynomials of degree $\leq3$.

\begin{lem}
\label{ThOrdProp}
A function $\a:(\Z_2)^n\to\Z_2$ is a polynomial of degree $\leq3$
if and only if $\d_3\a=0$.
\end{lem}

\begin{proof}
This is elementary, see also \cite{War,DV}.
\end{proof}

Proposition \ref{AlpDeg3} is proved.

\medskip

Let us now prove Corollary \ref{TriCol}.
If the map $\phi$ is symmetric, then Theorem \ref{AlphMainTh}
implies the existence of the generating function $\a$.
The map $\phi$ is then given by (\ref{Genalp2}).
One checks by an elementary calculation that 
$$
\phi(x+y,z,t)+\phi(x,z,t)+\phi(y,z,t)=\d_3\a(x,y,z,t).
$$
By Lemma \ref{KeyLem}, one has  $\d_3\a=0$.
It follows that $\phi$ is trilinear.

Furthermore, from \eqref{PhiBet}, we deduce that $\phi$ is alternate.

Corollary \ref{TriCol} is proved.

%%%%%%%%%%%%%%%%%%%%%%%%%%%%%%%%%
\subsection{Uniqueness of the generating function}\label{Unisex}
%%%%%%%%%%%%%%%%%%%%%%%%%%%%%%%%%

Let us show that there is a canonical way to choose a generating function.

\begin{lem}
\label{unipro}
(i)
Given a $(\Z_2)^n$-graded quasialgebra $\A$ with a generating function,
one can choose the generating function in such a way that it
satisfies
\begin{equation}
\label{NormAlp}
\left\{
\begin{array}{rl}
\a(0)=0,&\\[4pt]
\a(x)=1, & |x|=1.
\end{array}
\right.
\end{equation}

(ii)
There exists a unique generating function of $\A$ satisfying (\ref{NormAlp}).
\end{lem}

\begin{proof}
Part (i).
Every generating function $\a$ vanishes on the zero element
$0=(0,\ldots,0)$, cf. Section~\ref{DefGenFn}.
Furthermore, a generating function corresponding to a given algebra $\A$,
is defined up to a 1-cocycle on $(\Z_2)^n$.
Indeed, the functions $\b=\d\a$ and $\phi=\d_2\a$
that define the quasialgebra structure do not change if one adds a 1-cocycle to $\a$.
Since every 1-cocycle is a linear function, we obtain
$$
\a(x)\sim\a(x)+\sum_{1\leq{}i\leq{}n}\l_i\,x_i.
$$
One therefore can normalize $\a$ in such a way that $\a(x)=1$ for all $x$
such that $|x|=1$.

Part (ii).
The generating function normalized in this way is unique.
Indeed, any other function, say $\a'$, satisfying (\ref{NormAlp}) differs from $\a$
by a polynomial of degree $\geq2$, so that $\a-\a'$ cannot be a 1-cocycle.
Therefore, $\b'\not=\b$ which means the quasialgebra structure is different.
\end{proof}

We will assume the normalization (\ref{NormAlp}) in the sequel, whenever we
speak of \textit{the} generating function corresponding to a given algebra.

Let us now consider an algebra $\A$ with $n$ generators $u_1,\ldots,u_n$.
The group of permutations $\gS_n$ acts on $\A$ by permuting the generators.

\begin{cor}
\label{SimLem}
If the group of permutations $\gS_n$ acts on $\A$ by automorphisms,
then the corresponding generating function $\a$ is $\gS_n$-invariant.
\end{cor}

\begin{proof}
Let $\a$ be a generating function.
Since the algebra $\A$ is stable with respect to the $\gS_n$-action,
the function $\a\circ{}\s$ is again a generating function.
If, moreover, $\a$ satisfies~(\ref{NormAlp}), then
$\a\circ{}\s$ also satisfies this condition.
The uniqueness Lemma \ref{unipro} implies that
$\a\circ{}\s=\a$.
\end{proof}

Note that the converse statement holds in the complex case,
but fails in the real case.

%%%%%%%%%%%%%%%%%%%%%%%%%%%%%%%%%
\subsection{From the generating function to the twisting function}\label{FroSec}
%%%%%%%%%%%%%%%%%%%%%%%%%%%%%%%%%

Given an arbitrary polynomial map $\a:(\Z_2)^n\to\Z_2$ of $\deg\a\leq3$
such that $\a(0)=0$,
there is a simple way to associate a twisting function $f$
such that $(\bbK[(\Z_2)^n],f)$ admits $\a$ as a generating function.

\begin{prop}
\label{XXProp}
There exists a twisting function $f$
satisfying the property
\begin{equation}
\label{XXAlpha}
f(x,x)=\a(x).
\end{equation}
\end{prop}

\begin{proof}
Let us give an explicit formula for a twisting function $f$.
The procedure is linear, we associate to every monomial in $\a$
a function in two variables via the following rule:
\begin{equation}
\label{Fakset}
\begin{array}{rcl}
x_ix_jx_k&\longmapsto&x_ix_jy_k+x_iy_jx_k+y_ix_jx_k\\[4pt]
x_ix_j&\longmapsto&x_iy_j,\\[4pt]
x_i& \longmapsto & x_iy_i
\end{array}
\end{equation}
where $i<j<k$.
\end{proof}

%%%%%%%%%%%%%%%%%%%%%%%%%%%%%%%%%
%%%%%%%%%%%%%%%%%%%%%%%%%%%%%%%%%
\section{Proof of the simplicity criterion}\label{ProoSimProp}
%%%%%%%%%%%%%%%%%%%%%%%%%%%%%%%%%
%%%%%%%%%%%%%%%%%%%%%%%%%%%%%%%%%

In this section, we prove Theorem \ref{SimProp}.
We use the notation $\A$ to refer to any of the algebras
$\bbO_n, \bbO_{p,q}$ and $\bbM_n, \bbM_{p,q}$.

%%%%%%%%%%%%%%%%%%%%%%%%%%%%%%%%%
\subsection{The idea of the proof}\label{IdeaSect}
%%%%%%%%%%%%%%%%%%%%%%%%%%%%%%%%%
Our proof of simplicity of a  twisted group algebra $\A$ will
be based on the following lemma.

\begin{lem}\label{ASimple}
If for every homogeneous element $u_x$ in $\A$
there exists an element $u_y$ in $\A$ such that $u_x$ and $u_y$
anticommute, then $\A$ is simple.
\end{lem}

\begin{proof}
Let us suppose that there exists a nonzero proper two-sided ideal $\I$ in $\A$.
Every element $u$ in $\I$ is a linear combination of some homogeneous elements of $\A$. 
We write
$$
u=\l_1\,u_{x_1}+\cdots+\l_k\,u_{x_k}.
$$
Among all the elements of $\I$ we choose an element such that the number
$k$ of homogeneous components is the smallest possible. 
We can assume that $k\geq 2$, otherwise $u$ is homogeneous and therefore
$u^2$ is non-zero and proportional to $1$, so that $\I=\A$. 
In addition,  up to multiplication by $u_{x_1}$ and scalar normalization 
we can assume that 
$$
u=1+\l_2\,u_{x_2}+\cdots+\l_k\,u_{x_k}.
$$
If there exists an element $u_y\in \A$ anticommuting with $u_{x_2}$
then one obtains that $u\cdot{}u_y-u_y\cdot{}u$ is a nonzero element in $\I$ 
with a shorter decomposition into homogeneous components.
This is a contradiction with the choice of $u$.
Therefore, $\A$ has no proper ideal. 
\end{proof}

We now need to study central elements in $\A$, i.e., the elements commuting 
with every element of $\A$.

%%%%%%%%%%%%%%%%%%%%%%%%%%%%%%%%%
\subsection{Central elements}
%%%%%%%%%%%%%%%%%%%%%%%%%%%%%%%%%
In this section we study the \textit{commutative center} $\Zc(\A)$ of $\A$, i.e.,
$$
\Zc(\A)=\{w\in \A| \;w\cdot a=a\cdot w, \;\hbox{for all}\; a\in \A\}.
$$
Note that, in the case where $\A$ admits a generating function,
formula \eqref{PhiBet} implies that the commutative center is contained in the
associative nucleus of~$\A$, so that the commutative center
coincides with the usual notion of center, see \cite{ZSS}, p. 136
for more details.

The unit $1$ of $\A$ is obviously an element of the center.
We say that $\A$ has a trivial center if $\Zc(\A)=\bbK\,1$.

Consider the following particular element
$$
z=(1, \ldots, 1)
$$
in $(\Z_2)^n$, with all the components equal to $1$,
and the associated homogeneous element $u_z$ in $\A$.

\begin{lem}
\label{zcentral}
The element $u_z$ in $\A$ is central if and only if 
\begin{enumerate}
\item
$n=4m$ in the cases $\A=\bbO_n, \bbO_{p,q}$;
\item 
$n=4m+2$ in the cases $\A=\bbM_n, \bbM_{p,q}$.

\end{enumerate}
\end{lem}

\begin{proof} 
The element $u_z$ in $\A$ is central if and only if for all $y\in(\Z_2)^n$ one has $\b(y,z)=0$.
We use the generating function $\a$. 
Recall that $\b(y,z)=\a(y+z)+\a(y)+\a(z)$. 
The value $\a(x)$ depends only on the weight $|x|$, see Table \eqref{GenFuncTab}.
For every $y$ in  $\Z_2^n$, one has
$$
|z+y|=|z|-|y|.
$$

Case (1).
According to the Table \eqref{GenFuncTab}, one has $\a(x)=0$ if and only if
$|x|$ is a multiple of 4.

Assume $n=4m$. One gets $\a(z)=0$ and for every $y$ one has $\a(y)=0$ if and only if $\a(y+z)=0$. 
So, in that case, one always has  $\a(y)=\a(y+z)$ and therefore $\b(y,z)=0$.

Assume $n=4m+r$, $r=1,2$ or $3$. 
We can always choose an element $y$ such that $|y|=|r-2|+1$. We get
$$
\a(z)=\a(y)=\a(y+z)=1.
$$
Hence, $\b(y,z)=1$. This implies that $u_z$ is not central.

Case (2).
 According to the Table \eqref{GenFuncTab}, one has $\a(x)=0$  if and only if
$|x|$ is not equal to $1\mod4$.
 
 Assume $n=4m+2$. 
 One gets $\a(z)=0$ and for every $y$ one has $|y|=1 \mod 4$  if and only if $|y+z|=1 \mod 4$. 
 So, in that case, one always has  $\a(y)=\a(y+z)$ and therefore $\b(y,z)=0$.
 
Assume $n=4m+r$, $r=0,1$ or $3$. 
We choose the element $y=(1,0,\ldots,0)$, if $r=0,3$, or $y=(1,1,0,\ldots,0)$, if $r=1$. 
We easily compute
$
\b(y,z)=1.
$
This implies that $u_z$ is not central.
\end{proof}

Let us consider the case where $u_z$ is not central.

\begin{lem}
\label{nocentre}
If $u_z$ is not central, then $\A$ has a trivial center. 
\end{lem}
\begin{proof}
It suffices to prove that for every homogeneous element $u_x$ in $\A$,
that is not proportional to $1$, there exists an element $u_y$ in $\A$, 
such that $u_x$ and $u_y$ anticommute.
Indeed, if $u$ is central, then each homogeneous component of $u$ is central.

Let us fix $x\in (\Z_2)^n$ and the corresponding homogeneous element $u_x\in \A$,
such that $x$ is neither $0$, nor $z$. 
We want to find an element $y\in (\Z_2)^n$ such that $\b(x,y)=1$ or equivalently 
$u_x$ anticommutes with $u_y$.
Using the invariance of the functions $\a$ and $\b$ under permutations of the coordinates, 
we can assume that $x$ is of the form 
$$
x=(1,\ldots,1,0, \ldots, 0),
$$
where first $|x|$ entries are equal to 1 and the last entries are equal to 0.  
We assume $0<|x|<n$, so that, $x$ starts with 1 and ends by 0.

%We use the values of  the function $\a$ given in Table \eqref{GenFuncTab}.

{\bf Case $\A=\bbO_n$ or $\bbO_{p,q}$}.
If $|x|\not=4\ell$, then 
we use exactly the same arguments as in the proof of Lemma \ref{zcentral}
in order to find a suitable $y$
(one can also take one of the elements $y=(1,0,\cdots, 0)$ or $y=(0,\cdots, 0, 1)$). 
Assume $|x|=4\ell$.
Consider the element
$$
y=(0,1, \ldots,1,0,\ldots,0),
$$
with $|y|=|x|$.
One has $\a(x)=\a(y)=0$ and $\a(x+y)=1$. So we also have $\b(x,y)=1$ and deduce $u_x$ anticommutes with $u_y$.

\textbf{Case $\A=\bbM_n$ or $\bbM_{p,q}$}.
Similarly to the proof of Lemma \ref{zcentral}, 
if $k\not=4\ell+2$ then we can find a $y$ such that $u_y$ anticommutes with $u_x$.
If $k=4\ell+2$ then $\a(x)=0$. 
The element $y=(0,\cdots, 0, 1)$ satisfies $\a(y)=1$ and $\a(x+y)=0$.
\end{proof}

Consider now the case where $u_z$ is a central element.
There are two different possibilities: ${u_z}^2=1$, or ${u_z}^2=-1$.

\begin{lem}
\label{ADecomp}
If $u_z\in\A$ is a central element and if ${u_z}^2=1,$
then the algebra splits into a direct sum of two subalgebras:
$$
\A=\A^+\oplus \A^-,
$$
where $\A^+:=\A\cdot(1+u_z)$ and $\A^-:=\A\cdot(1-u_z)$.
\end{lem}

\begin{proof}
Using ${u_z}^2=1,$ one immediately obtains
\begin{equation}
\label{calcuz}
\begin{array}{rcl}
(1\pm u_z)^2&=&2\,(1\pm u_z),\\[4pt]
(1+u_z)\cdot(1-u_z)&=&0.
\end{array}
\end{equation}
In addition, using the expression of $\phi$ in terms of $\b$ given in \eqref{PhiBet}
and the fact that $\b(\cdot, z)=0$, one deduces that $\phi(\cdot,\cdot,z)=0$ and thus
$a\cdot(b\cdot{}u_z)=(a\cdot{}b)\cdot{}u_z$ 
for all $a,b\in \A$. 
It follows that 
\begin{equation}
\label{assoz}
\left(a \cdot(1\pm u_z)\right)
\cdot
\left(b \cdot(1\pm u_z)\right)=
(a \cdot b) \cdot\left((1\pm u_z) \cdot(1\pm u_z)\right)
\end{equation}
for all $a,b\in \A$. 
This expression, together with the above computations \eqref{calcuz},
shows that  $\A^+$ and $\A^-$ are, indeed, two subalgebras of $\A$ and that they
satisfy $\A^+\cdot\A^-=\A^-\cdot \A^+=0$.
Moreover, for any $a\in \A$, one can write
$$
a=\textstyle \half\,a \cdot(1+u_z) +\half\,a \cdot(1-u_z).
$$
This implies the direct sum decomposition $\A=\A^+\oplus\A^-$.
\end{proof}

Notice that the elements $\half\,(1+u_z)$ and $\half\,(1-u_z)$
are the units of $\A^+$ and $\A^-$, respectively.

%%%%%%%%%%%%%%%%%%%%%%%%%%%%%%%%%
\subsection{Proof of Theorem \ref{SimProp}, part (i)}
%%%%%%%%%%%%%%%%%%%%%%%%%%%%%%%%%

If $n\not=4m$, then by Lemma \ref{ASimple} and Lemma \ref{nocentre} we immediately deduce
 that $\bbO_n$ is simple.
 
If $n=4m$, then $u_z$ is central and, in the complex case, one has
$u_z^2=1$.
By Lemma \ref{zcentral} and Lemma \ref{ADecomp}, we immediately deduce
 that $\bbO_n$ is not simple and one has
 $$
 \bbO_{4m}=\bbO_{4m}\cdot(1+u_z)\oplus \bbO_{4m}\cdot(1-u_z),
 $$
where $z=(1,\ldots,1)\in (\Z_2)^n$.
It remains to show that the algebras $\bbO_{4m-1}$ 
and $ \bbO_{4m}\cdot(1\pm u_z)$  are isomorphic.
Indeed, using the computations \eqref{calcuz} and \eqref{assoz}, one checks that
 the map
 $$
 u_x\longmapsto  \textstyle \half\,u_{(x,0)}\cdot(1\pm u_z),
 $$
 where $x\in(\Z_2)^{n-1}$, is the required isomorphism.
 
The proof in the case of $\bbM_n$ is completely similar.

 %%%%%%%%%%%%%%%%%%%%%%%%%%%%%%%%%
\subsection{Proof of Theorem \ref{SimProp}, part (ii)}\label{Thmii}
%%%%%%%%%%%%%%%%%%%%%%%%%%%%%%%%%
The algebras $\bbO_{p,q}$ with $p+q\not=4m$
and the algebras $\bbM_{p,q}$ with $p+q\not=4m+2$
are simple because their complexifications are.

If now $u_z$ is central, then the property $u_z^2=1$ or $-1$ becomes crucial.
Using the expressions for $f_{\bbO}$ or $f_{\bbM}$, one computes
\begin{eqnarray*}
f_{\bbO_{p,q}}(z,z)&=&
\sum_{i<j<k}z_iz_jz_k \quad  
+\sum_{i\leq j} z_iz_j
+\sum_{1\leq{}i\leq{}p}z_i \\
&=&\dfrac{n(n-1)(n-2)}{6}+ \dfrac{n(n+1)}{2}+p\\
&=&p,\mod2.
\end{eqnarray*}
And similarly, one obtains
$
f_{\bbM_{p,q}}(z,z)=p.
$
It follows that $u_z^2=(-1)^p$. 

If $p$ is even, then Lemma \ref{zcentral} just applied
guarantees that $\A$ is not simple.

Finally, if $u_z$ is central and $p$ is odd, then $u_z^2=-1$.

\begin{lem}
\label{Moulinette}
If $u_z$ is central and $p$ is odd, then
$$
\bbO_{p,q}\cong\bbO_{p,q-1}\otimes\C,
\qquad
\bbM_{p,q}\cong\bbM_{p,q-1}\otimes\C.
$$
\end{lem}
\begin{proof}
We construct an explicit isomorphism from $\bbO_{p,q-1}\otimes\C$ to $\bbO_{p,q}$ as follows.
\begin{eqnarray*}
u_x\otimes 1&\longmapsto& u_{(x,0)}\\
u_x\otimes \sqrt{-1} &\longmapsto& u_{(x,0)}\cdot u_z\;,
\end{eqnarray*}
for all $x\in (\Z_2)^{n-1}$.
We check that the above map is indeed an isomorphism of algebras by noticing that 
$f_{\bbO_{p,q}}((x,0),(y,0))=f_{\bbO_{p,q-1}}(x,y)$.
\end{proof}

Let us show that Lemma \ref{Moulinette} implies that the
(real) algebras $\bbO_{p,q}$ with $p+q=4m$ and $p$ odd
and the algebras $\bbM_{p,q}$ with $p+q=4m+2$  and $p$ odd are
 simple.
 Indeed, 
 $$
 \bbO_{p,q-1}\otimes_\R\C\cong\bbO_{p+q-1}
 $$
 viewed as a real algebras.
 We then use the following well-known fact.
 A simple unital complex algebra viewed as a real algebra remains simple.

 The proof of Theorem \ref{SimProp} is complete.
 
 Lemma \ref{Moulinette} also implies Corollary \ref{IsoPr}.

%%%%%%%%%%%%%%%%%%%%%%%%%%%%%%%%%
%%%%%%%%%%%%%%%%%%%%%%%%%%%%%%%%%
\section{Hurwitz-Radon square identities}\label{Sqrd}
%%%%%%%%%%%%%%%%%%%%%%%%%%%%%%%%%
%%%%%%%%%%%%%%%%%%%%%%%%%%%%%%%%%

In this section, we use the algebras $\bbO_n$ (and, in the real case, $\bbO_{0,n}$) to
give explicit formul{\ae} for solutions of a classical problem
of products of squares.
Recall, that the octonion algebra is related to the 8-square identity.
In an arbitrary commutative ring,
the product $(a_1^2+\cdots{}+a_8^2)\,(b_1^2+\cdots{}+b_8^2)$
is again a sum of 8 squares $c_1^2+\cdots{}+c_8^2$,
where $c_k$ are explicitly given by bilinear forms in $a_i$ and $b_j$
with coefficients $\pm1$, see, e.g., \cite{ConSmi}.
This identity is equivalent to the fact that $\bbO$ is a composition algebra,
that is, for any $a,b\in\bbO$, the norm of the product is equal to the product of the norms:
\begin{equation}
\label{NormPr}
\cN(a\cdot{}b)=\cN(a)\,\cN(b).
\end{equation}
Hurwitz proved that there is no similar $N$-square identity for $N>8$,
as there is no composition algebra in higher dimensions.

The celebrated Hurwitz-Radon Theorem \cite{Hur,Rad} establishes the maximal
number $r$, as function of $N$, such that there exists an identity
\begin{equation}
\label{Radon}
\left(a_1^2+\cdots{}+a_N^2\right)
\left(b_1^2+\cdots{}+b_r^2\right)=
\left(c_1^2+\cdots{}+c_N^2\right),
\end{equation}
where $c_k$ are bilinear forms in $a_i$ and $b_j$.
The theorem states that $r=\rho(N)$ is the maximal number,
where $\rho(N)$ is the Hurwitz-Radon function defined as follows.
Write $N$ in the form $N=2^{4m+ \ell}\,N'$, where $N'$ is odd and $\ell=0,1,2$ or $3$, then
$$
\rho(N):=8m+2^\ell.
$$
It was proved by Gabel \cite{Gab} that the bilinear forms $c_k$ can be chosen with
coefficients~$\pm1$.
Note that the only interesting case is $N=2^n$ since the general case is an immediate
corollary of this particular one.
We refer to \cite{Squa,Sha} for the history,
results and references.

In this section, we give explicit formul{\ae} for the solution to the Hurwitz-Radon equation,
see also \cite{LMO} for further development within this framework.

%%%%%%%%%%%%%%%%%%%%%%%%%%%%%%%%%
\subsection{The explicit solution}\label{Expl}
%%%%%%%%%%%%%%%%%%%%%%%%%%%%%%%%%

We give explicit solution for Hurwitz-Radon equation \eqref{Radon}
for any $N=2^n$ with $n$ not a multiple of 4.

We label the $a$-variables and the $c$-variables by elements of $(\Z_2)^n$.
In order to describe the labeling of the $b$-variables,
we consider the following particular elements of $(\Z_2)^n$:
 \begin{equation*}
\begin{array}{rcl}
e_0&:=& (0,0,\ldots,0),\\[4pt]
\overline{e_0}&:=& (1,1,\ldots,1),\\[4pt]
e_i&:=&(0,\ldots,0,1,0,\ldots,0), \text{ where 1 occurs at the $i$-th position,}\\[4pt]
\overline{e_i}&:=&(1,\ldots,1,0,1,\ldots,1), \text{ where 0 occurs at the $i$-th position,}
\end{array}
\end{equation*}
 for all $1\leq i\leq n$ and $1<j\leq n$.
We then introduce the following subset $H_n$ of $(\Z_2)^n$:
\begin{equation}\label{defHn}
\begin{array}{rcl}
H_n&=& \{ e_i, \overline{e_i}, \; 1\leq i\leq n\}, \text{ for } n=1 \mod 4,\\[6pt]
H_n&=& \{ e_i, e_1+e_j,\; 0\leq i\leq n,\; 1<j\leq n\}, \text{ for } n=2 \mod 4,\\[6pt]
H_n&=& \{ e_i, \overline{e_i},\; 0\leq i\leq n\}, \text{ for } n=3 \mod 4.\\[4pt]
\end{array}
\end{equation}
In each case, the subset $H_n$ contains exactly $\r(2^n)$ elements. 

We write the Hurwitz-Radon identity in the form
$$
\Big( \sum_{x\in (\Z_2)^n} a_x^2\;\Big)\Big( \sum_{x\in H_n} b_x^2\;\Big)
= \sum_{x\in (\Z_2)^n} c_x^2.
$$
We will establish the following.

\begin{thm}
\label{Solcx}
The bilinear forms
\begin{equation}
\label{ExplSolRad}
c_x=\sum_{y\in H_n}(-1)^{f_{\bbO}(x+y,y)}\,a_{x+y}b_y,
\end{equation}
where $f_{\bbO}$ is the twisting function of the algebra $\bbO_n$ defined in  \eqref{NashProd},
are a solution to the Hurwitz-Radon identity.
\end{thm}

In order to prove Theorem \ref{Solcx} we will need to define
the natural norm on $\bbO_n$.

%%%%%%%%%%%%%%%%%%%%%%%%%%%%%%%%%
\subsection{The Euclidean norm}\label{Norm}
%%%%%%%%%%%%%%%%%%%%%%%%%%%%%%%%%

Assume that a twisted group algebra $\A=(\bbK[(\Z_2)^n],f)$
is equipped with a generating function $\a$.
Assume furthermore that the twisting function satisfies $f(x,x)=\a(x)$, as in \eqref{XXAlpha}.

The involution on $\A$ is defined for every
$a=\sum_{x\in(\Z_2)^n}a_x\,u_x$, where $a_x\in\C$ (or in~$\R$)
are scalar coefficients and $u_x$ are the basis elements,
by the formula
$$
\bar{a}=\sum_{x\in(\Z_2)^n}(-1)^{\a(x)}\,a_x\,u_x.
$$
We then define the following norm of an element $a\in\A$:
$$
\cN(a):=\left(a\cdot{}\bar{a}\right)_0.
$$

\begin{prop}
\label{ObvPr}
The above norm is nothing but the
Euclidean norm in the standard basis:
\begin{equation}
\label{EuclEq}
\cN(a)=\sum_{x\in(\Z_2)^n}a_x^2.
\end{equation}
\end{prop}
\begin{proof}
One has: 
$$
\cN(a)=\sum{}(-1)^{\a(x)}\,a_x^2\,u_x\cdot{}u_x
=\sum{}(-1)^{\a(x)+f(x,x)}\,a_x^2.
$$
The result then follows from the assumption $f(x,x)=\a(x)$.
\end{proof}

The following statement is a general criterion for $a,b\in\A$ to satisfy the
composition equation (\ref{NormPr}).
This criterion will be crucial for us to establish the square identities.

\begin{prop}
\label{EuclProp}
Elements $a,b\in\A$ satisfy (\ref{NormPr}), if and only if
for all $x,y,z,t\in(\Z_2)^2$ such that
$$
x+y+z+t=0,
\qquad
(x,y)\not=(z,t),
\qquad
a_x\,b_y\,a_z\,b_t\not=0,
$$
one has $\a(x+z)=\a(y+t)=1$.
\end{prop}
\begin{proof}
Calculating the left-hand-side of (\ref{NormPr}), we obtain
$$
\cN(a\cdot{}b)=
\sum_{x+y+z+t=0}(-1)^{f(x,y)+f(z,t)}\,a_x\,b_y\,a_z\,b_t
$$
According to (\ref{EuclEq}), the product of the norm in the right-hand-side is: 
$$
\cN(a)\,\cN(b)=\sum_{x,y}\,a_x^2\,b_y^2.
$$
It follows that the condition (\ref{NormPr}) is satisfied if and only if
$$
f(x,y)+f(z,t)+f(x,t)+f(z,y)=1,
$$
whenever $(x,y)\not=(z,t)$ and $a_x\,b_y\,a_z\,b_t\not=0$.

Taking into account the linearity of the function (\ref{Fakset})
and substituting $t=x+y+z$, one finally gets (after cancellation):
$$
f(z,x)+f(x,z)+f(x,x)+f(z,z)=1.
$$
In terms of the function $\a$ this is exactly the condition $\a(x+z)=1$.
Hence the result.
\end{proof}

%%%%%%%%%%%%%%%%%%%%%%%%%%%%%%%%%
\subsection{Proof of Theorem \ref{Solcx}}\label{Choose}
%%%%%%%%%%%%%%%%%%%%%%%%%%%%%%%%%

Let us apply Proposition \ref{EuclProp} to the case of the algebra $\bbO_n$.

Given the variables $(a_x)_{x\in (\Z_2)^n}$ and $(b_x)_{x\in H_n}$,
where $H_n$ is the subset defined in \eqref{defHn}, form the following vectors in $\bbO_n$,
$$
a=\sum_{x\in (\Z_2)^n} a_x\,u_x, 
\qquad  b=\sum_{y\in H_n} b_y\,u_y.
$$
Taking two distinct elements $y,t\in H_n$ one always has $\a_{\bbO}(y+t)=1$. 
Therefore, from Proposition \ref{EuclProp} one deduces that 
$\cN(a)\cN(b)=\cN(a\cdot{}b).$
Writing this equality in terms of coordinates of the three elements
$a, b$ and $c=a\cdot{}b$, one obtains the result.

Theorem \ref{Solcx} is proved.

\medskip

Let us give one more classical identity that can be realized in the algebra $\bbO_n$.

\begin{ex}
The most obvious choice of two elements $a,b\in \bbO_n$ 
that satisfy the condition~(\ref{NormPr})
is: $a=a_0\,u_0+\sum{}a_i\,u_i$ and $b=b_0\,u_0+\sum{}b_i\,u_i$.
One immediately obtains in this case the following elementary but elegant identity:
$$
(a_0^2+\cdots+a_n^2)\,(b_0^2+\cdots+b_n^2)=
(a_0\,b_0+\cdots+a_n\,b_n)^2+
\sum_{0\leq{}i<j\leq{}n}(a_i\,b_j-b_j\,a_i)^2,
$$
for an arbitrary $n$, known as the Lagrange identity.
\end{ex}

%%%%%%%%%%%%%%%%%%%%%%%%%%%%%%%%%
%%%%%%%%%%%%%%%%%%%%%%%%%%%%%%%%%
\section{Relation to code loops}\label{LaSec}
%%%%%%%%%%%%%%%%%%%%%%%%%%%%%%%%%
%%%%%%%%%%%%%%%%%%%%%%%%%%%%%%%%%

The constructions of the algebras that we use in this work
are closely related to some constructions in the theory of Moufang Loops. 
In particular, they lead to examples of Code Loops \cite{Gri}. 
In this section, we apply our approach in order to obtain an explicit construction of the famous Parker Loop. 

\bigskip

%%%%%%%%%%%%%%%%%%%%%%%%%%%%%%%%%
\paragraph{\bf The loop of the basis elements.}
%%%%%%%%%%%%%%%%%%%%%%%%%%%%%%%%%
The structure of loop is a nonassociative version of a group
(see, e.g., \cite{Goo}).

\begin{prop}
\label{MouPr}
The basis elements together with their opposites, $\{\pm u_x, x\in (\Z_2)^n\}$,
in a twisted algebra $(\tA,f)$, form a loop with respect to the multiplication rule.
Moreover, this loop is a Moufang loop whenever $\phi=\d{}f$ is symmetric.
\end{prop}
\begin{proof}
The fact that the elements $\pm u_x$ form a loop is evident.
If the function $\phi=\d f$ is symmetric,
then this loop satisfies the Moufang identity:
$$
u\cdot (v \cdot (u \cdot w))= ((u\cdot v) \cdot u) \cdot w
$$
for all $u,v,w$.
Indeed, the symmetry of $\phi$ implies that $\phi $ is also trilinear and alternate,
see Corollary \ref{TriCol}. 
\end{proof}

Let us mention that the Moufang loops associated with the octonions and
split-octonions are important classical objects invented by Coxeter \cite{Cox}.

\bigskip

%%%%%%%%%%%%%%%%%%%%%%%%%%%%%%%%%
\paragraph{\bf Code loops.}
%%%%%%%%%%%%%%%%%%%%%%%%%%%%%%%%%
The notion of code loops has been introduced by Griess, \cite{Gri}. We recall the construction and main results. A doubly even binary code is a subspace $V$ in $(\Z_2)^n$ such that any vectors in $V$ has weight a multiple of 4. It was shown that there exists a function~ $f$ from $V\times V $ to $\Z_2$, called a \textit{factor set} in \cite{Gri}, satisfying

\begin{enumerate}
\item $f(x,x)=\textstyle\frac{1}{4}|x|$,

\item $f(x,y)+f(y,x)=\half |x\cap y|$,

\item $\d f(x,y,z) = |x\cap y\cap z|$,
\end{enumerate}
where $|x\cap y|$ (resp. $|x\cap y\cap z|$) is the number of nonzero coordinates  in both $x$ and $y$ (resp. all of $x, y, z$).
The associated code loop $\Lc (V)$ is the set $\{\pm u_x, x\in V\}$ together with the multiplication law
$$
u_x\cdot u_y=(-1)^{f(x,y)}\,u_{x+y}.
$$
The most important example of code loop, is the Parker loop that plays an important r\^ole
in the theory of sporadic finite simple groups. 
The Parker loop is the code loop obtained from the Golay code. 
This code  can be described as the 12-dimensional subspace of $(\Z_2)^{24}$ 
given as the span of the rows of the following matrix, see \cite{ConSlo},

\begin{tiny}
\begin{equation*}
G=\left(
\begin{array}{cccccccccccccccccccccccc}
1&&&&&&&&&&&&1&0&1&0&0&0&1&1&1&0&1&1\\
&1&&&&&&&&&&& 1&1&0&1&0&0&0&1&1&1&0&1\\
&&1&&&&&&&&&& 0&1&1&0&1&0&0&0&1&1&1&1\\
&&&1&&&&&&&&& 1&0&1&1&0&1&0&0&0&1&1&1\\
&&&&1&&&&&&&& 1&1&0&1&1&0&1&0&0&0&1&1\\
&&&&&1&&&&&&& 1&1&1&0&1&1&0&1&0&0&0&1\\
&&&&&&1&&&&&& 0&1&1&1&0&1&1&0&1&0&0&1\\
&&&&&&&1&&&&& 0&0&1&1&1&0&1&1&0&1&0&1\\
&&&&&&&&1&&&& 0&0&0&1&1&1&0&1&1&0&1&1\\
&&&&&&&&&1&&& 1&0&0&0&1&1&1&0&1&1&0&1\\
&&&&&&&&&&1&& 0&1&0&0&0&1&1&1&0&1&1&1\\
&&&&&&&&&&&1& 1&1&1&1&1&1&1&1&1&1&1&0\\
\end{array}
\right)
\end{equation*}
\end{tiny}

\bigskip

%%%%%%%%%%%%%%%%%%%%%%%%%%%%%%%%%
\paragraph{\bf An explicit formula for the Parker loop.}
%%%%%%%%%%%%%%%%%%%%%%%%%%%%%%%%%

Let us now give the generating function of the Parker loop.
We identify the Golay code with the space $(\Z_2)^{12}$, in such a way
that the $i$-th row of the matrix $G$, denoted $\ell_i$, is identified with the
$i$-th standard basic vector
$e_i=(0, \ldots, 0, 1, 0 \ldots,0)$ of $(\Z_2)^{12}$.
As previously, we write $u_i=u_{e_i}=u_{\ell_i}$ the corresponding vector in the Parker loop.
The coordinates of an element $x\in(\Z_2)^{12}$ are denoted by
$(x_1,\ldots,x_{11},x_{12})$. 

\begin{prop}
The Parker loop is given by the following generating function $\a$ from 
$(\Z_2)^{12}$ to $\Z_2$.
\begin{equation}
\label{PLGF}
\begin{array}{rcl}
\a_G(x)&=&\displaystyle
\sum_{1\leq i \leq 11 }x_ix_{i+1}\left(x_{i+5}+x_{i+8}+x_{i+9}\right)
+x_ix_{i+2}\,(x_{i+6}+x_{i+8})\\[12pt]
&&\displaystyle
\; +\; x_{12}\,\Big(\sum_{1\leq i \leq 11}x_i +\sum_{1\leq i <j\leq 11}x_ix_j\Big),
\end{array}
\end{equation}
where the indices of $x_{i+k}$ are understood modulo 11.
\end{prop}
\begin{proof}
The ternary function 
$$
\phi(x,y,z)=\d{}f(x,y,z)= |x\cap y\cap z|
$$
is obviously symmetric in $x,y,z$.
Theorem \ref{AlphMainTh} then implies the existence of a
generating function $\a_G$. 
By Proposition \ref{AlpDeg3} we know that $\a_G$ is a polynomial
of degree $\leq3$. 
Moreover, linear terms in $\a_G$ do not contribute in the quasialgebra structure
($i.e$ do not contribute in the expression of $\b$ and $\phi$, see \eqref{Genalp1}).
To determine the quadratic and cubic terms, we use the following equivalences
$$
\a_G \text{ contains the term } x_ix_j, i\not=j
\Longleftrightarrow u_i, u_j \text{ anti-commute,}
$$
$$
\a_G \text{ contains the term } x_ix_jx_k, i\not=j\not=k
\Longleftrightarrow u_i, u_j, u_k \text{ anti-associate}.
$$
For instance, the construction of the Parker loop gives that 
$u_i$ and $u_j$ commute for all $1\leq i,j \leq11$, since
$|\ell_{i}\cap \ell_j|=8$, for $1\leq i,j \leq11$. 
Thus, $\a_G$ does not contain any of the quadratic terms $x_ix_j$,  $1\leq i\not=j \leq11$.
But, $u_{12}$ anti-commutes with $u_i$, $i \leq11$, 
since $|\ell_{12}\cap \ell_i|=6$, $i \leq11$. 
So that  the terms $x_{12}x_i$, $ i \leq11$, do appear in the expression of $\a_G$.
Similarly, one has to determine which one of the triples $u_i, u_j, u_k$ anti-associate
to determine the cubic terms in $\a_G$. 
This yields to the expression \eqref{PLGF}
\end{proof}

The explicit formula for the factor set $f$ in coordinates on $(\Z_2)^{12}$
is immediately obtained by (\ref{Fakset}).
Note that the signature in this case is $(11,1)$, so that we have to add
$x_{12}y_{12}$ to~(\ref{Fakset}).

\begin{rem}
The difference between the loops generated by the basis elements of
$\bbO_n$ and $\bbM_n$ and the Parker loop is that the function
\eqref{PLGF} is not $\gS_n$-invariant.
Our classification results cannot be applied in this case.
\end{rem}

We hope that the notion of generating function can be a useful
tool for study of code loops.

%%%%%%%%%%%%%%%%%%%%%%%%%%%%%%%%%
%%%%%%%%%%%%%%%%%%%%%%%%%%%%%%%%%
\section{Appendix: linear algebra and differential calculus over $\Z_2$}
%%%%%%%%%%%%%%%%%%%%%%%%%%%%%%%%%
%%%%%%%%%%%%%%%%%%%%%%%%%%%%%%%%%

The purpose of this Appendix is to relate the algebraic problems we
study to the general framework of linear algebra over $\Z_2$
which is a classical domain.

\bigskip

%%%%%%%%%%%%%%%%%%%%%%%%%%%%%%%%%
\paragraph{\bf Automorphisms of $(\Z_2)^n$ and linear equivalence.}
%%%%%%%%%%%%%%%%%%%%%%%%%%%%%%%%%

All the algebraic structures on $(\Z_2)^n$ we consider are invariant
with respect to the action of the group automorphisms
$$
\mathrm{Aut}((\Z_2)^n)\cong\GL(n,\Z_2).
$$
For instance, the generating function $\a:(\Z_2)^n\to\Z_2$, as well as
$\b$ and $\phi$, are considered up to the $\mathrm{Aut}((\Z_2)^n)$-equivalence
(called ``congruence'' in the classic literature \cite{Alb}).

\bigskip

%%%%%%%%%%%%%%%%%%%%%%%%%%%%%%%%%
\paragraph{\bf Quadratic forms.}
%%%%%%%%%%%%%%%%%%%%%%%%%%%%%%%%%

The interest to describe an algebra in terms of a generating function
can be illustrated in the case of the Clifford algebras.

There are exactly  two non-equivalent
non-degenerate quadratic forms on $(\Z_2)^{2m}$ with coefficients in $\Z_2$
(see \cite{Alb,Dieu} for the details):
\begin{equation}
\label{Darboux}
\a(x)=
x_1x_{m+1}+\cdots+x_mx_{2m}+\l\,(x_m^2+x_{2m}^2),
\end{equation}
where $\l=0$ or $1$.
Note that sometimes the case $\l=1$ is not considered (see \cite{KMRT}, p.xix)
since the extra term is actually linear, for $x_i^2=x_i$.
The corresponding polar bilinear form $\b=\d\a$
and the trilinear form $\phi=\d_2\a$ do not depend on $\l$.
The corresponding twisted group algebra is isomorphic to the
Clifford algebra $\Cl_n$.

The normal form (\ref{Darboux}) is written in the standard Darboux basis,
this formula has several algebraic corollaries.
For instance, we immediately obtain the well-known factorization of the complex Clifford algebras:
$$
\Cl_{2m}\cong
\Cl_2^{\otimes{}m}\cong
\C[2^m],
$$
where $\C[2^m]$ are $(2^m\times2^m)$-matrices.
Indeed, the function (\ref{Darboux}), with $\l=0$, is nothing
but the sum of $m$ generating functions of $\Cl_2$.
The other classical symmetry and periodicity theorems for the Clifford algebras
can also be deduced in this way.

Let us mention that bilinear forms over $\Z_2$ is still an interesting subject~\cite{Leb}.

\bigskip

%%%%%%%%%%%%%%%%%%%%%%%%%%%%%%%%%
\paragraph{\bf Cubic polynomials.}
%%%%%%%%%%%%%%%%%%%%%%%%%%%%%%%%%

In this paper, we were led to consider polynomials $\a:(\Z_2)^n\to\Z_2$ of degree~3:
$$
\a(x)=\sum_{i<j<k}\a^3_{ijk}\,x_ix_jx_k+
\sum_{i<j}\a^2_{ij}\,x_ix_j,
$$
where $\a^3_{ijk}$ and $\a^2_{ij}$ are arbitrary coefficients
(equal to 1 or 0).
It turns out that it is far of being obvious to understand what means $\a$ is ``non-degenerate''.

To every polynomial $\a$, we associate a binary function $\b=\d\a$ and a trilinear form
$\phi=\d_2\a$, see formula (\ref{Genalp2}),
which is of course just the polarization (or linearization) of~$\a$.
The form $\phi$ is alternate:
$\phi(x,x,.)=\phi(x,.,x)=\phi(.,x,x)=0$ and depends only on the
homogeneous part of degree 3 of $\a$, i.e., only on $\a^3_{ijk}$.
There are three different ways to understand the notion of non-degeneracy.

(1)
The most naive way: $\a$ (and $\phi$) is non-degenerate
if for all linearly independent $x,y\in(\Z_2)^n$, the linear
function $\phi(x,y,.)\not\equiv0$.
One can show that, with this definition,
{\it  there are no non-degenerate cubic forms on $(\Z_2)^{n}$
for $n\geq3$.}
This is of course not the way we take.

(2)
The second way to understand non-degeneracy is as follows.
The trilinear map $\phi$ itself defines an $n$-dimensional algebra.
Indeed, identifying $(\Z_2)^n$ with its dual space,
the trilinear function $\phi$ 
defines a product $(x,y)\mapsto\phi(x,y,.)$.
One can say that $\phi$ (and $\a$) is non-degenerate
if this algebra is simple.
This second way is much more interesting and is related to
many different subjects.
For instance, classification of simple Lie (super)algebras over $\Z_2$
is still an open problem, see \cite{BGL} and references therein.
This definition also depends only on the
homogeneous part of degree 3 of $\a$.

(3)
We understand non-degeneracy yet in a different way.
We say that $\a$ is non-degenerate if for
all linearly independent $x,y$ there exists $z$ such that
$$
\b(x,z)\not=0,
\qquad
\b(y,z)=0,
$$
where $\b=\d\a$.
This is equivalent to the fact that the algebra with the generated
function~$\a$ is simple, cf. Section~\ref{ProoSimProp}.

We believe that every non-degenerate (in the above sense) polynomial of degree 3
on~$(\Z_2)^{n}$ is equivalent to one of the two forms (\ref{NashAlp}) and (\ref{NashAlpBis}).
Note that a positive answer would imply the uniqueness results of Section \ref{UniqReSec}
without the $\gS_n$-invariance assumption.

\bigskip

%%%%%%%%%%%%%%%%%%%%%%%%%%%%%%%%%
\paragraph{\bf Higher differentials.}
%%%%%%%%%%%%%%%%%%%%%%%%%%%%%%%%%

Cohomology of abelian groups with coefficients in $\Z_2$ is
a well-known and quite elementary theory explained in many
textbooks.
Yet, it can offer some surprises.

Throughout this work, we encountered
and extensively used the linear operators $\d_k$,
for $k=1,2,3$, cf. (\ref{Genalp2}) and~(\ref{DelTrequat}),
that associate a $k$-cochain on $(\Z_2)^n$ to a function.
These operators were defined in \cite{War}, 
and used in the Moufang loops theory, \cite{Gri,DV,NV}.
Operations of this type are called \textit{natural} or \textit{invariant}
since they commute with the action of $\mathrm{Aut}((\Z_2)^n)$.
The operator $\d_k$ fits to the usual understanding of ``higher derivation'' since
the condition $\d_k\a=0$ is equivalent to the fact that $\a$
is a polynomial of degree $\leq{}k$.

The cohomological meaning of $\d_k$ is as follows.
In the case of an abelian group $G$, the cochain complex with coefficients in $\Z_2$
has a richer structure.
There exist $k$ natural operators acting from $C^k(G;\Z_2)$ to $C^{k+1}(G;\Z_2)$ :

 \SelectTips{eu}{12}%
 \xymatrix{
&C^1(G;\Z_2)\ar@{->}[rr]^\d
&&C^2(G;\Z_2) \ar@<6pt>[rr]^{\d_{1,0}} \ar@<-8pt>[rr]^{\d_{0,1}}
&&C^3(G;\Z_2) \ar@<13pt>[rr]^{\d_{1,0,0}} \ar@<-13pt>[rr]^{\d_{0,0,1}} \ar@{->}[rr]^{\d_{0,1,0}}&&
\;\cdots\\ 
}
\medskip
\noindent
where $\d_{0,\ldots,1,\ldots,0}$ is a ``partial differential'', i.e., the differential with respect to
one variable.
For instance, if $\b\in C^2(G;\Z_2)$ is a function in two variables, then
$$
\d_{1,0}\b(x,y,z)=\b(x+y,z)+\b(x,z)+\b(y,z).
$$
In this formula $z$ is understood as a parameter and one can write
$\d_{1,0}\b(x,y,z)=\d\g(x,y)$, where $\g=\b(.,z)$.
At each order one has
$$
\d=\d_{1,0,\ldots,0}+\d_{0,1,\ldots0}+\cdots+\d_{0,\ldots,0,1}.
$$

If $\a\in{}C^1(G;\Z_2)$, then an \textit{arbitrary} sequence of the partial derivatives
gives the same result: $\d_k\a$, for example one has
$$
\d_2\a=\d_{1,0}\circ\d\a=\d_{0,1}\circ\d\a,
\qquad
\d_3\a=\d_{1,0,0}\circ\d_{1,0}\circ\d\a=\cdots=\d_{0,0,1}\circ\d_{0,1}\circ\d\a,
$$
etc.
The first of the above equations corresponds to the formula
\eqref{PhiBet} since $\b=\d\a$.

\bigskip

%%%%%%%%%%%%%%%%%

\end{document}